\documentclass[12pt]{amsart}
\usepackage{amssymb}
\usepackage{mathtools}
\usepackage[english]{babel}
\usepackage{epsfig}
\usepackage{mathptmx}
\usepackage{amsmath,amsfonts,amssymb}
\usepackage{mathtools}
\usepackage{mathrsfs}
\usepackage{graphicx}
\usepackage{latexsym}
\usepackage{verbatim}
\usepackage{enumerate}
\usepackage[svgnames]{xcolor}
\usepackage{ulem}
\usepackage{cancel}
\usepackage{tabularx}
\usepackage{caption}
\usepackage[font=small]{subcaption}
\usepackage{tikz-cd}
\tikzcdset{every label/.append style = {font = \normalsize}}
\usepackage{pifont}

\usepackage{float}

\usepackage{glossaries}
\setglossarystyle{long4colheader}

\usepackage[backref=page]{hyperref} 
\usepackage{wasysym}
\usepackage{soul}
\newcommand{\hide}[1]{}
\numberwithin{equation}{section}

\def\spa{\mathop{\rm span}}
\def\conf{
          \mathop{{\rm Conf}_2(\hat{\C})}}

\newcommand{\D}{\mathbb D}

\newcommand{\R}{\mathbb R}

\newcommand{\C}{\mathbb C}

\newcommand{\A}{\mathcal A}

\newcommand{\Aut}{{\sf Aut}(\mathbb D)}
\newcommand{\Rot}{{\sf Rot}}

\newcommand{\N}{\mathbb N}

\def\id{{\sf id}}

\def\Aut{{\sf Aut}}

\def\1#1{\overline{#1}}
\def\2#1{\widetilde{#1}}
\def\3#1{\widehat{#1}}
\def\4#1{\mathbb{#1}}
\def\5#1{\frak{#1}}
\def\6#1{{\mathcal{#1}}}

\DeclareMathOperator{\Rudin}{R}

\newcommand{\mcite}[1]{\csname b@#1\endcsname}

\newtheoremstyle{break}
  {8pt}{8pt}%
  {\itshape}{}%
  {\bfseries}{}%
  {\newline}{}%
\theoremstyle{break}

\theoremstyle{theorem}

\def\id{{\sf id}}

\def\Aut{{\sf Aut}}

\newcommand{\at}[2][\big]{#1\vert_{#2}}
\newcommand{\argument}{\,\cdot\,}

\DeclarePairedDelimiter{\abs}{\lvert}{\rvert}
\DeclarePairedDelimiter{\norm}{\lVert}{\rVert}
\newcommand{\del}{\mathop{}\!\partial}

\newcommand{\Schwarzian}{\operatorname{S}}

\makeatletter
\def\blfootnote{\xdef\@thefnmark{}\@footnotetext}
\makeatother

\allowdisplaybreaks

\theoremstyle{break}

\newtheorem{theorem}{Theorem}[section]
\newtheorem*{theorem*}{Theorem}
\newtheorem{lemma}[theorem]{Lemma}
\newtheorem*{lemma*}{Lemma}
\theoremstyle{break}
\newtheorem{proposition}[theorem]{Proposition}
\newtheorem{corollary}[theorem]{Corollary}

\theoremstyle{break}
\newtheorem{definition}[theorem]{Definition}

\theoremstyle{break}
\newtheorem{remark}[theorem]{Remark}

\numberwithin{equation}{section}

\newmuskip\pFqmuskip

\newcommand*\pFq[6][8]{%
	\begingroup 
	\pFqmuskip=#1mu\relax
	\mathchardef\normalcomma=\mathcode`,
	\mathcode`\,=\string"8000
	\begingroup\lccode`\~=`\,
	\lowercase{\endgroup\let~}\pFqcomma
	{}_{#2}F_{#3}{\left[\genfrac..{0pt}{}{#4}{#5};#6\right]}%
	\endgroup
}
\newcommand{\pFqcomma}{{\normalcomma}\mskip\pFqmuskip}

\newcommand*\FF[3][8]{%
	\begingroup 
	\pFqmuskip=#1mu\relax
	\mathchardef\normalcomma=\mathcode`,
	\mathcode`\,=\string"8000
	\begingroup\lccode`\~=`\,
	\lowercase{\endgroup\let~}\pFqcomma
	{}_{3}F_{2}{\left[\genfrac..{0pt}{}{#2}{#3}\right]}%
	\endgroup
}

\newenvironment{mylist}{\begin{list}{}%
{\labelwidth=2em\leftmargin=\labelwidth\itemsep=.4ex plus.1ex minus.1ex\topsep=.7ex plus.3ex minus.2ex}%
\let\itm=\item\def\item[##1]{\itm[{\rm ##1}]}}{\end{list}}

\makeglossaries

\newglossaryentry{C^*}{name=\ensuremath{\C^*},
                       description={\ensuremath{\C \setminus\{0\}}}}
\newglossaryentry{hatC^*}{name=\ensuremath{\hat{\C}^*},
                          description={\ensuremath{\hat{\C} \setminus\{0\}}}}
\newglossaryentry{D^{m,n}}{name=\ensuremath{D^{m,n}},
                           description={Peschl--Minda operator}}

\newglossaryentry{D_z^n}{name=\ensuremath{D_z^n},
                          description={pure Peschl--Minda operator}}

\newglossaryentry{D_w^n}{name=\ensuremath{D_w^n},
                           description={pure Peschl--Minda operator}}

\newglossaryentry{D_*^{m,n}}{name=\ensuremath{D_*^{m,n}},description={dual Peschl--Minda operator}}

\newglossaryentry{R}{name=\ensuremath{\Rudin},
                           description={Rudin operator}}

\newglossaryentry{del^{m,n}}{name=\ensuremath{\partial^m_1 \partial^n_2},
                             description={Wirtinger derivatives}}
\newglossaryentry{C^{infty}}{name=\ensuremath{C^{\infty}(U)},
                        description={the set of $C^{\infty}$ functions on $U$}}
\newglossaryentry{SmoothLocal}{name=\ensuremath{C^{\infty}_{\text{loc}}},description={We write $f\in C^{\infty}_{\text{loc}}(V)$ if $f \in C^{\infty}(U)$ for some open subset $U$ of $V$.}}

\newglossaryentry{Hol}{name=\ensuremath{\mathcal{H}(U)},description={the set of holomorphic functions on $U$}}

\newglossaryentry{HolLocal}{name=\ensuremath{\mathcal{H}_{\text{loc}}},description={We write $f\in \mathcal{H}_{\text{loc}}(V)$ if $f \in \mathcal{H}(U)$ for some open subset $U$ of $V$.}}

\newglossaryentry{Delta_{zw}}{name=\ensuremath{\Delta_{zw}},description={Laplacian on $\Omega$}}

\author[M. Heins]{Michael Heins}
\address{M. Heins: Department of Mathematics, University of W\"urzburg, Emil Fischer Strasse 40, 97074, W\"urzburg, Germany.} \email{michael.heins@mathematik.uni-wuerzburg.de}

\author[A.~Moucha]{Annika Moucha$^\S$}
\address{A. Moucha: Department of Mathematics, University of W\"urzburg, Emil Fischer Strasse 40, 97074, W\"urzburg, Germany.} \email{annika.moucha@uni-wuerzburg.de}

\author[O. Roth]{Oliver Roth}
\address{O. Roth: Department of Mathematics, University of W\"urzburg, Emil Fischer Strasse 40, 97074, W\"urzburg, Germany.} \email{roth@mathematik.uni-wuerzburg.de}

\title[Function Theory off the complexified unit circle]{Function Theory off the complexified unit circle:\\[2mm] Fr\'echet space structure and automorphisms}
\thanks{$^\S\,$Partially supported by the Alexander von Humboldt Stiftung}

\begin{document}

\maketitle
\footnotetext[1]{By convention,  $z \cdot \infty=\infty \cdot z=\infty$ for all $z \in \hat{\C}\setminus \{0\}$, while $0 \cdot \infty=\infty \cdot 0=1$.}
\blfootnote{\textit{Key words and phrases.} Schauder basis, Invariant Laplacian, Conformal invariance \hfill{\today}}


\begin{abstract}
Motivated by recent work on strict deformation quantization of the unit disk and the Riemann sphere,
 we study the Fr\'echet space structure of the set of holomorphic functions on the complement $\Omega:=\{(z,w)\in \hat{\mathbb{C}}^2\, :\, z\cdot w\not=1\}$ of the complexified unit circle ${\{(z,w) \in \hat{\mathbb{C}}^2 \, : \, z\cdot w=1\}}$. We also characterize the subgroup of all biholomorphic automorphisms of $\Omega$ which leave the canonical Laplacian on $\Omega$ invariant.
  \end{abstract}

\section{Introduction}

We work on the Riemann sphere $\hat{\C}:=\C \cup \{\infty\}$.
The \textit{complement of the complexified unit circle} is defined\footnotemark[1]
as the open subset
\begin{equation*} \label{eq:defOmega}
\Omega:=\left\{ (z,w) \in \hat{\C}^2 \, : \, z \cdot w\not=1\right\} \,
\end{equation*}
of $\hat{\C}^2$. The Fr\'echet space $\mathcal{H}(\Omega)$ of all holomorphic functions $f : \Omega \to \C$, equipped with its natural  topology of locally uniform convergence, plays a peculiar and rather puzzling  key role in recent work \cite{BW,EspositoSchmittWaldmann2019,KrausRothSchoetzWaldmann,SchmittSchoetz2022} on strict deformation quantization  of the unit disk $\D:={\{z \in \C \, : \, |z|<1\}}$ and the Riemann sphere $\hat{\C}$.
Moreover, it is shown in \cite{HeinsMouchaRoth2, Annika}, that the fine structure of $\mathcal{H}(\Omega)$ also provides an efficient device to investigate spectral properties of the invariant Laplacians on the disk and the sphere (see e.g.~Helgason \cite{Helgason70}, Rudin \cite{Rudin84}) 
 using methods from complex analysis.
 \medskip

 The purpose of the present paper is to analyze the structure of the Fr\'echet space $\mathcal{H}(\Omega)$ and to identify the group of biholomorphic automorphisms of~$\Omega$ which leave the canonical Laplacian of $\Omega$ invariant. In order to place the results of this paper into a broader context, we briefly indicate the role the Fr\'echet space $\mathcal{H}(\Omega)$ is playing in strict deformation quantization as well as for the spectral theory of the invariant Laplacian for the unit disk $\D$ and the sphere~$\hat{\C}$.
\begin{mylist}
\item[(i)] Roughly speaking, in the theory of strict deformation quantization, a star product $\star_{\hbar}$ of Wick--type is an associative, \textit{non--commutative} product or ``deformation'' of the pointwise product of functions defined on some complex manifold~$M$ and depending on a complex parameter~$\hbar$ s.t.~the product $\star_\hbar$ is compatible with the complex structure of $M$, see e.g.~\cite{Waldmann2019} for~a detailed survey.
It turns out, see \cite{BW,EspositoSchmittWaldmann2019,KrausRothSchoetzWaldmann}, that the vector spaces
$$ \mathcal{A}(\hat{\C}):=\left\{\hat{\C} \ni z \mapsto  f(z,-\overline{z}) \, : \, f \in \mathcal{H}(\Omega)\right\}$$
and
$$ \mathcal{A}(\D):=\left\{\D \ni z \mapsto f(z,\overline{z}) \, : \, f \in \mathcal{H}(\Omega)\right\} \, ,$$
which are subspaces of the real analytic functions on $\hat{\C}$ resp. $\D$, are of particular interest. In fact, they can be equipped with star products $\star_{\hbar,\hat{\C}}$ and $\star_{\hbar,\D}$ such that  $(\mathcal{A}(\hat{\C}),\star_{\hbar,\hat{\C}})$ and $(\mathcal{A}(\D),\star_{\hbar,\D})$ are non--commutative \textit{Fr\'echet algebras} w.r.t.~the subspace topology induced from the natural topology of $\mathcal{H}(\Omega)$.
In addition, these star products are holomorphic w.r.t.~$\hbar$, and they are {invariant} w.r.t.~the action of the groups $\text{Rot}(\hat{\C})$ and $\Aut(\D)$ of holomorphic isometries of $\hat{\C}$ resp.~${\D}$.

In \cite{HeinsMouchaRoth1} we go one step further and show that the star products $\star_{\hbar,\D}$ and $\star_{\hbar,\hat{\C}}$ are simply the restrictions of a naturally defined star product $\star_{\hbar}$ on the ``ambient'' Fr\'echet space $\mathcal{H}(\Omega)$.  In particular, $(\mathcal{H}(\Omega),\star_{\hbar})$ becomes a Fr\'echet algebra when equipped with its natural compact--open topology, and the dependence on the parameter $\hbar$ is holomorphic. This gives a natural interpretation of the striking similarities of $\mathcal{A}(\hat{\C})$ and $\mathcal{A}(\D)$ observed in \cite{EspositoSchmittWaldmann2019,SchmittSchoetz2022}. Moreover, the star product $\star_{\hbar}$ on $\mathcal{H}(\Omega)$ is fully invariant under the distinguished subgroup \begin{equation}\label{eq:OmegaAutomorphismIntro}
    \mathcal{M}
    :=
    \bigcup \limits_{ \psi \in \Aut(\hat{\C})}
    \left\{ (z,w) \mapsto
        \left(
    \psi(z),\frac{1}{\psi(1/w)}
    \right)\, , \,
        (z,w) \mapsto
        \left(
    \psi(w),\frac{1}{\psi(1/z)}
    \right) \right\}\,
\end{equation}
of automorphisms of $\Omega$, which corresponds to the full M\"obius group $\Aut(\hat{\C})$ of all M\"obius transformations.

\item[(ii)] While the group $\Aut(\Omega)$ of all biholomorphic automorphisms is exceedingly large, the distinguished subgroup $\mathcal{M}$ is characterized as the group of exactly those automorphisms of $\Omega$ which leave the natural Laplacian on $\Omega$ invariant, see Theorem \ref{thm:Invariance} in Section \ref{sec:Automorphisms} below.

\item[(iii)] The complement of the complexified unit circle has several other equivalent models such as the second configuration space ${\{(z,w) \in \hat{\C}^2 \, : \, w\not=z\}}$ of the Riemann sphere~$\hat{\C}$ or the complex two--sphere $\mathbb{S}^2_\C=\{(z_1,z_2,z_3) \in \C^3 \, : \, z_1^2+z_2^2+z_3^2=1\}$, see Section~\ref{sec:OtherModels} below. By the invariance properties of the star product $\star_{\hbar}$ on $\Omega$, one therefore obtains a star product on $\mathcal{H}(\mathbb{S}^2_\C)$, which is invariant under the group $\text{SO}(3,\C)$ of all complex linear transformations preserving the quadratic form $z_1^2+z_2^2+z_3^2$.
In the complex two--sphere model, the restriction process of $\Omega$ to its ``rotated diagonal'' $\{(z,-\overline{z}) \, : \, z \in \hat{\C}\} \subseteq \Omega$ reveals that, algebraically, the Fr\'echet algebra $\mathcal{A}(\hat{\C})$ is the same as the set of spherical harmonics on the \textit{real} two--sphere. This aspect will be further discussed in the forthcoming paper \cite{Annika} of the second--named author where a coherent theory of complexified spherical harmonics is developed.

\item[(iv)] The star product $\star_{\hbar}$ on $\mathcal{H}(\Omega)$ has an explicit and simple formula in terms of $\mathcal{M}$--invariant differential operators acting on $\mathcal{H}(\Omega)$, see \cite[Sec.~6]{HeinsMouchaRoth1}. These differential operators are natural extensions to $\Omega$ of conformally invariant differential operators acting on $\hat{\C}$ resp.~$\D$ of Peschl--Minda type, which have been continuously studied since the 1950s, see \cite{KS07diff,KS11,Peschl1955,Schippers2007,Sugawa2000}. By restriction to $\hat{\C}$ and $\D$ one obtains explicit formulas for the star products~$\star_{\hbar,\hat{\C}}$ on $\hat{\C}$ and $\star_{\hbar,\D}$ on $\D$ in terms of those Peschl--Minda differential operators, see again \cite[Sec.~6]{HeinsMouchaRoth1}.

\item[(v)] The Fr\'echet space $\mathcal{H}(\Omega)$ has an intriguing structure. It is the direct sum of two complementary subspaces $\mathcal{H}_-(\Omega)$ and $\mathcal{H}_+(\Omega)$ which we call the \textit{past} and the \textit{future} of $\mathcal{H}(\Omega)$, see Theorem \ref{thm:HOmegaDecomposition} below. These subspaces ultimately come from composition operators induced by two biholomorphic maps from $\C^2$ onto the open subsets
\begin{equation*} \label{eq:OmegaDecompositionIntro}
  \Omega_+:=\Omega \setminus \left\{(z,\infty) \, : \, z \in \hat{\C}\right\} \, , \qquad \Omega_-:=\Omega \setminus \left\{ (\infty,w) \, : \, w \in \hat{\C}\right\} \,
\end{equation*}
of $\Omega$, which are obtained from $\Omega$ by removing the horizontal resp.~vertical complex line through the point $(\infty,\infty) \in\Omega$, see Figure \ref{fig:OmegaPM} below. In particular, the Fr\'echet spaces $\mathcal{H}(\Omega_+)$ and $\mathcal{H}(\Omega_-)$ are both isomorphic to $\mathcal{H}(\C^2)$, the Fr\'echet space of all entire functions on~$\C^2$.
Theorem \ref{thm:HOmegaDecomposition2} below shows that $\mathcal{H}(\Omega_+)$ and $\mathcal{H}(\Omega_-)$ completely encode the Fr\'echet space structure of $\mathcal{H}(\Omega)$. As a byproduct, we get (see Corollary~\ref{cor:Schauder}) that the functions
\begin{equation*}\label{eq:SchauderBasisIntro}
 f_{p,q}(z,w)=\frac{z^pw^q}{\left(1-zw\right)^{\max\{p,q\}}}\, ,\quad  p,q\in\N_0  \, ,
\end{equation*}
form a Schauder basis for $\mathcal{H}(\Omega)$.
\item[(vi)] The Fr\'echet spaces $\mathcal{H}(\Omega_+)$ and $\mathcal{H}(\Omega_-)$ also play a crucial role for studying invariant differential operators on $\Omega$ and its projections, the unit disk $\D$ and the Riemann sphere~$\hat{\C}$. In fact, it is shown in \cite{HeinsMouchaRoth1} that $\mathcal{H}(\Omega_+)$ and $\mathcal{H}(\Omega_-)$ are identified as the ``natural'' function spaces on which the Peschl--Minda type differential operators (see (iv) above) operate. Moreover, the Fr\'echet spaces $\mathcal{H}(\Omega_+)$ and $\mathcal{H}(\Omega_-)$ are utilised in \cite{HeinsMouchaRoth2} as basic tools for
a function--theoretic characterization of the ``exceptional'' eigenspaces of the invariant Laplacian on the unit disk, which have been studied by Rudin \cite{Rudin84}.
\item[(vii)] In addition, the second--named author shows in \cite{Annika} that these ``exceptional'' eigenspaces span all of $\mathcal{H}(\Omega)$. In fact, much more is true: there is a ``natural'' Schauder basis for $\mathcal{H}(\Omega)$ consisting of eigenfunctions of the invariant Laplacian on $\Omega$. By restriction to $\hat{\C}$ and $\D$ one obtains explicit Schauder bases of the Fr\'echet algebras $\mathcal{A}(\hat{\C})$ and $\mathcal{A}(\D)$ consisting precisely of the eigenfunctions of the invariant Laplacian on $\hat{\C}$ resp.~the ``exceptional'' eigenfunctions of Rudin \cite{Rudin84} of the invariant Laplacian on $\D$. This provides a geometrically pleasing intrinsic description of the algebras $\mathcal{A}(\hat{\C})$ and $\mathcal{A}(\D)$.
\end{mylist}

In summary, the study of the set of holomorphic functions on  the complement of the com\-plexified unit circle provides a unified framework for investigating  conformally invariant differential operators on the disk and the sphere within their conjecturally natural habitat. In addition, our further investigations in \cite{HeinsMouchaRoth1,HeinsMouchaRoth2,Annika} offer a complex--analysis approach to Wick--type strict deformation quantization and the spectral properties of the invariant Laplacians on the unit disk, the Riemann sphere and the complex two--sphere as well as its various interactions that might otherwise  remain hidden.

\medskip
The paper is organized in the following way. We start with two preliminary sections. In Section \ref{sec:prelim} we briefly discuss  the canonical Laplacian of $\Omega$ and its main properties.  In Section \ref{sec:OtherModels} we construct
 two other very useful,  biholomorphically equivalent models of $\Omega$,
 the second configuration space $\conf=\{(z,w) \in \hat{\C}^2 \, : \, z\not=w\}$ of $\hat{\C}$ and  the complex two--sphere $\mathbb{S}_{\C}^2=\{(z_1,z_2,z_3) \in \C^3 \, : \, z_1^2+z_2^2+z_3^2=1\}$.
 After these preparations, we turn to a discussion of the main results of this paper.
In Section \ref{sec:schauder} we describe in detail the fine structure of the Fr\'echet space $\mathcal{H}(\Omega)$. Section \ref{sec:Automorphisms} is concerned with the group $\Aut(\Omega)$ of all biholomorphic automorphisms of~$\Omega$. In particular, we identify the M\"obius  subgroup $\mathcal{M}$ of $\Aut(\Omega)$ as the set of those biholomorphisms of $\Omega$ which leave the Laplacian of $\Omega$ invariant. In Section \ref{sec:RiemannianMetric} we introduce a canonical holomorphic Riemannian metric on $\Omega$ and study its relation to the M\"obius subgroup~$\mathcal{M}$. 
The proofs of the results of Section \ref{sec:schauder} are given in Section \ref{sec:proof}, while those of Sections~\ref{sec:Automorphisms} and \ref{sec:RiemannianMetric} are presented in Section \ref{sec:proof2}.

\medskip

We finally note that there is a natural higher dimensional analogue for the complement of the complexified unit circle, namely the complement  of the complexified unit ball in $\C^N$,
$$ \Omega_N:=\left\{ (z_1,\ldots, z_N,w_1,\ldots, w_N) \in \hat{\C}^{2N}\, : \, z_1 w_1+\ldots +z_N w_N\not=1\right\} \, ,$$
and one can define a  star product on the unit ball of $\C^N$ by working on $\Omega_N$ instead of $\Omega=\Omega_1$, see  \cite{KrausRothSchoetzWaldmann}. Nevertheless, we have decided to present our work exclusively in the $N=1$ dimensional case. In fact, some of our results easily extend to  the $N$--dimensional situation, essentially without extra work except for adding more variables, while others conjecturally do not. Perhaps  ``this is a feature, not a bug'' \cite[p.~xiv]{Hubbard}.

\section{The invariant Laplacian on $\Omega$} \label{sec:prelim}

The set $\Omega$ is a complex manifold of complex dimension $2$. Its complex structure is given by just two charts:
  \begin{itemize}
  \item[(i)] \textit{Standard chart}
  \begin{equation*}
      \label{eq:ChartStandard}
      \phi_{+} : \Omega \cap( \C\times\C) \to \C^2\, , \qquad \phi_+(u,v)=(u,v) \, .
  \end{equation*}
\item[(ii)] \textit{Flip chart}
  \begin{equation*}
      \phi_- : \Omega \cap (\hat{\C} \setminus\{0\} \times \hat{\C} \setminus \{0\} ) \to \C^2 \, , \qquad \phi_-(u,v):=(1/v,1/u) \, .
  \end{equation*}
\end{itemize}

In local coordinates the \textit{canonical Laplacian} $\Delta_{zw}$ on $\Omega$ is defined by
\begin{equation*}\label{eq:Laplace}
	\gls{Delta_{zw}}
    F(z,w)
    \coloneqq
    4(1-zw)^2
    \del_z \del_w
    F(z,w)
\end{equation*}
for $z,w \in \C$. Here, $\partial_z$ and $\partial_w$ denote the complex Wirtinger derivatives for a function $F$ which is supposed to be twice partially differentiable in the real sense and defined in an open  neighborhood $U \subseteq \Omega$ of $(z,w)$.
We are mostly interested in the case when $F$ belongs to the set $\mathcal{H}(U)$ of  holomorphic functions on $U$.
It is not difficult to check that $\Delta_{zw}$ is a well--defined continuous linear operator acting on $\mathcal{H}(U)$ for every open subset $U$ of $\Omega$. The canonical Laplacian $\Delta_{zw}$ on~$\Omega$ can be viewed as  the complex counterpart of both the hyperbolic Laplacian $\Delta_{\D}$  on the unit disk $\D$ defined by
$$ \Delta_{\D} f(z):=\left(1-|z|^2 \right)^2 \Delta f(z) \, , \quad z=x+i y \in \D \, , $$
as well as of the spherical Laplacian $\Delta_{\hat{\C}}$ on the sphere $\hat{\C}$ defined in local coordinates by
$$ \Delta_{\hat{\C}} f(z):=\left(1+|z|^2 \right)^2 \Delta f(z) \, , \quad z=x+i y \in \C \, .$$
Here,
$$\Delta:=\frac{\partial^2}{\partial x^2}+\frac{\partial^2}{\partial y^2}$$
is the standard Euclidean Laplacian. It is an elementary fact that the hyperbolic Laplacian $\Delta_{\D}$ is invariant under the group $\Aut(\D)$ of all conformal automorphisms $T : \D \to \D$ of $\D$. This means that $\Delta_{\D} (f \circ T)=\Delta_{\D} f \circ T$ for all $T \in \Aut(\D)$ and all $C^2$--functions $f : U \to \D$ defined on some open set $U \subseteq \D$. In a similar way, the spherical Laplacian $\Delta_{\hat{\C}}$ is invariant under the group $\Rot(\hat{\C})$ of holomorphic rigid motions of $\hat{\C}$. Both $\Aut(\D)$ and $\Rot(\hat{\C})$ are subgroups of the group $\Aut(\hat{\C})$ of all M\"obius transformations, that is, the biholomorphic self--maps $\psi : \hat{\C} \to \hat{\C}$. Clearly, every M\"obius transformation $\psi \in \Aut(\hat{\C})$ gives rise to the following two biholomorphic self--maps of~$\Omega$:
$$   (z,w) \mapsto
        \left(
    \psi(z),\frac{1}{\psi(1/w)}
    \right)\, , \quad
        (z,w) \mapsto
        \left(
    \psi(w),\frac{1}{\psi(1/z)}
  \right) \, ,$$
  and the set of all such automorphisms of $\Omega$ is a subgroup of $\Aut(\Omega)$, which we denote by~$\mathcal{M}$. Slightly abusing language, we call $\mathcal{M}$ the \textbf{M\"obius subgroup of $\Aut(\Omega)$} and its elements \textbf{M\"obius automorphisms of $\Omega$}. Just as $\Aut(\D)$ and $\Rot(\hat{\C})$ operate transitively on $\D$ resp.~$\hat{\C}$, the M\"obius subgroup $\mathcal{M}$ acts transitively on $\Omega$. Moreover, the canonical Laplacian $\Delta_{zw}$ is invariant under every $T \in \mathcal{M}$:
  $$ \Delta_{zw} \left(F \circ T \right)=\Delta_{zw} \left(F\right) \circ T $$
  for every $F \in \mathcal{H}(\Omega)$. For this reason we call $\Delta_{zw}$ the \textbf{invariant Laplacian on $\Omega$}.

  \section{Other models of $\Omega$} \label{sec:OtherModels}

While for applications to deformation quantization and to the spectral property of the invariant Laplacian on the unit disk, the set $\Omega$ seems most suitable, some of the proofs become more transparent in the $\conf$--setting, and the connections to spherical harmonics are best discussed in the $\mathbb{S}^2_\C$--model.
We therefore briefly discuss these two other biholomorphically equivalent models for the manifold~$\Omega$.

\subsection{The second configuration space of the Riemann sphere} \label{sub:conf}

The mapping
$$ \mathcal{T} (z,w) :=\left(z, 1/w \right)$$
defines a biholomorphic map from $\Omega$ onto
$$ G=\left\{ (z,w) \in \hat{\C} \, : \, z\not=w\right\} \, , $$
the so--called \textit{second configuration space of the sphere $\hat{\C}$}, see \cite{Coh}.
In the literature, the second configuration space of $\hat{\C}$ is often denoted by $\conf$, but for ease of notation we write $G$ instead of $\conf$. Clearly, $G$ is an open subset of $\hat{\C}^2$ and pathconnected. It is known that $G=\conf$ is homotopy equivalent to $\hat{\C}$, see \cite[Example~2.4 (1)]{Coh}. Thus its fundamental group $\pi_1(G)$ is trivial and $G$ is simply connected.
  Moreover, it is easy to see that every entire function $g : \C \to \C$ gives rise to the following biholomorphic self--map of $G$:
  $$ (z,w) \mapsto \left(z+g\left(\frac{1}{z-w}\right), w+ g\left(\frac{1}{z-w}\right)\right) \, .$$
In particular, the automorphism group $\Aut(G)$ of $G$ is infinite--dimensional.
  By way of the biholomorphic map $\mathcal{T} : \Omega \to G$, these properties of $G$ also hold for $\Omega$. Thus $\Omega$ is simply connected and $\Aut(\Omega)$ is infinite--dimensional.

\subsection{The complex two--sphere} \label{sub:complexsphere}

Another useful model for $\Omega$ is the complex two--sphere
 $$\mathbb{S}_\C^2 = \left\{(z_1,z_2,z_3) \in \C^3 \;\big|\; z_1^2 + z_2^2 + z_3^2 = 1\right\} \, .$$
 In fact, $\Omega$ is biholomorphically equivalent  to $\mathbb{S}_\C^2$
 by way of the biholomorphic map $S \colon \Omega \longrightarrow \mathbb{S}_{\C}^2$,
    \begin{equation*}
        \label{eq:Stereographic}
        S(z,w)
        =
        \begin{cases}\displaystyle
         \left( \frac{z - w}{1 - zw}, -i \frac{z + w}{1 - zw}, -\frac{1 + z w}{1 - z w} \right) & \phantom{ if \quad } (z,w) \in \C^2,\, zw\not=1, \\[4mm]
        \displaystyle  \left(1/z ,i/z, 1 \right) & \text{ if \quad } z \in \hat{\C} \setminus \{0\}, \,w=\infty,\\[4mm]
   \displaystyle      \left( -1/w, i/w,1 \right) & \phantom{ if \quad } w \in \hat{\C} \setminus \{0\}, \,z=\infty.
       \end{cases}
    \end{equation*}
The inverse map of  $S \colon \Omega \longrightarrow \mathbb{S}_{\C}^2$ is given by $\pi \colon \mathbb{S}_{\C}^2 \to \Omega$,
   \begin{equation*}
        \pi(z_1,z_2,z_3)
        =
        \begin{cases}\displaystyle
         \left( \frac{z_1+i z_2}{1-z_3}, - \frac{z_1 -i z_2}{1-z_3} \right) & \phantom{ if \quad } z_3 \not=1,  \\[4mm]
        \displaystyle  \left(1/z_1,\infty \right) & \text{ if \quad } z_3=1, \,z_2=i z_1,\\[4mm]
   \displaystyle      \left( \infty, -1/z_1  \right)& \phantom{ if \quad } z_3=1, \,z_2=-i z_1.
       \end{cases}
     \end{equation*}
     It is natural to call $\pi \colon \mathbb{S}_{\C}^2 \to \Omega$ the ``complex stereographic projection'', since if we denote by $$  \mathbb{S}_{\R}^2 \coloneqq \mathbb{S}^2_{\C} \cap \R^3 = \big\{ (x_1, x_2, x_3) \in \R^3 \;|\; x_1^2 + x_2^2 + x_2^3 = 1\big\}$$
     the Euclidean two--sphere in $\R^3$, then
    $$
    \pi(x_1,x_2,x_3)=\left( \frac{x_1+i x_2}{1-x_3}, -\frac{x_1-i x_2}{1-x_3} \right) \quad \text{ for all } (x_1,x_2,x_3) \in \mathbb{S}_{\R}^2 \, $$
    and
    $$ \mathbb{S}^2_{\R} \to \hat{\C} \, , \qquad (x_1,x_2,x_3) \mapsto  \frac{x_1+i x_2}{1-x_3}$$
is the standard stereographic projection of $\mathbb{S}^2_{\R}$ onto $\hat{\C}$. In particular,
$\pi$ maps $\mathbb{S}_{\R}^2$ onto the ``rotated diagonal'' $\{(z,-\overline{z}) \, : z \in \hat{\C}\}\subseteq \Omega$. Furthermore, if $H_{\R}^2$ denotes the image of the hyperboloid $\{ (x_1,x_2,x_3) \in \R^3 \, : \, -x_1^2-x_2^2+x_3^2=1\}$ under the biholomorphic map $(z_1,z_2,z_3) \mapsto (iz_1,iz_2,z_3)$, then $\pi$ maps the lower  half $\{(x_1,x_2,x_3) \in H_{\R}^2  \, : \,  x_3 \le -1\}$
    onto $\{(z,\overline{z})\} \, : \, z \in \D\}$ and the upper half $\{(x_1,x_2,x_3) \in H_{\R}^2  \, : \,  x_3 \ge 1\}$
    onto $\{(z, \overline{z})\} \, : \, z \in \hat{\C} \setminus \overline{\D}\}$.

\begin{remark}\label{rem:orthogonalmatrices}
On the complex two--sphere $\mathbb{S}^2_{\C}$, the subgroup $\mathcal{M}$ of $\Aut(\Omega)$ defined by (\ref{eq:OmegaAutomorphismIntro})  corresponds to
$$\operatorname{SO}(3,\C):=\left\{R \in \C^{3 \times 3} \, : \, R^T R=I, \, \det R=1\right\}
\, .$$
\end{remark}

\section{The structure  of the Fr\'echet space  $\mathcal{H}(\Omega)$ } \label{sec:schauder}

In this section we describe the fine structure of the Fr\'echet space $\mathcal{H}(\Omega)$. In particular, we
shall see that the functions
\begin{equation}\label{eq:SchauderBasis}
 f_{p,q} : \Omega \to \C \, , \qquad f_{p,q}(z,w)=\frac{z^pw^q}{\left(1-zw\right)^{\max\{p,q\}}}\, ,\quad  p,q\in\N_0  \, ,
\end{equation}
are a Schauder basis for $\mathcal{H}(\Omega)$. As immediate corollaries we obtain that the functions ${z \mapsto f_{p,q}(z,\overline{z})}$ are a Schauder basis for $\A(\D)$ and the functions $z\mapsto f_{p,q}(z,-\overline{z})$ are a Schauder basis for $\A(\hat{\C})$. These corollaries  have already been proven in \cite[Theorem 3.16]{KrausRothSchoetzWaldmann} and \cite[Proposition 6.4]{SchmittSchoetz2022}. Our approach is much more conceptional, and has various useful rami\-fications for studying invariant differential operators  and the spectral theory of the invariant Laplacian on~$\Omega$ for which we refer to \cite{HeinsMouchaRoth1, HeinsMouchaRoth2}. The basic tool is a natural decomposition of $\mathcal{H}(\Omega)$
  into a direct sum of two closed subspaces $\mathcal{H}_+(\Omega)$ and $\mathcal{H}_-(\Omega)$, which are defined as follows. We denote by $\spa M$  the closure of all finite linear combinations of a subset $M$ of a Fr\'echet space~$X$.

\begin{definition}[The past and the future in $\mathcal{H}(\Omega)$] \label{def:pastfuture}
  The closed subspace
  $$ \mathcal{H}_-(\Omega):=\spa \left\{ f_{p,q} \, : \, 0 \le q \le p<\infty\right\} $$
  is called the \textbf{past} in $\mathcal{H}(\Omega)$ and the closed subspace
  $$ \mathcal{H}_+(\Omega):=\spa \left\{ f_{p,q} \, : \, 0 \le p < q<\infty\right\} $$
 is called the \textbf{future} in $\mathcal{H}(\Omega)$.
\end{definition}

\begin{theorem}[Canonical Decomposition of $\mathcal{H}(\Omega)$] \label{thm:HOmegaDecomposition}
$\mathcal{H}(\Omega)$ is the topological direct sum of the past $\mathcal{H}_-(\Omega)$ and the  future $\mathcal{H}_+(\Omega)$,
  \begin{equation} \label{eq:HOmegaDecomposition}
    \mathcal{H}(\Omega)=\mathcal{H}_-(\Omega) \oplus \mathcal{H}_+(\Omega) \, .
  \end{equation}
\end{theorem}

  The decomposition (\ref{eq:HOmegaDecomposition}) is intimately tied to the  decomposition of the domain $\Omega$,
  $$ \Omega=\Omega_+ \cup \Omega_- \cup \left\{(\infty,\infty)\right\} $$
  into the two subdomains
\begin{equation*} \label{eq:OmegaDecomposition}
  \Omega_+:=\Omega \setminus \left\{(z,\infty) \, : \, z \in \hat{\C}\right\} \, , \qquad \Omega_-:=\Omega \setminus \left\{ (\infty,w) \, : \, w \in \hat{\C}\right\} \, .
  \end{equation*}
  Note that $\Omega_- \cap \Omega_+=\Omega \cap \C^2$ and
  $$ \mathcal{H}(\Omega_-) \cap \mathcal{H}(\Omega_+)=\mathcal{H}(\Omega) \, .$$
  Moreover, it is an immediate consequence of \eqref{eq:SchauderBasis} and Definition \ref{def:pastfuture} that $$\mathcal{H}_\pm(\Omega)\subseteq\mathcal{H}(\Omega_\pm)\, ,$$
  and the inclusion is strict. The subdomains $\Omega_+$ and $\Omega_-$ of $\Omega$ are visualized in Figure~\ref{fig:OmegaPM}. The edges of each square represent points near infinity. If the edge belongs to the domain, it is dashed. The blue dots correspond to boundary points. Note that points on opposite edges and in particular the four corners are identified.
  \begin{figure}[h]
        \begin{minipage}{6\textwidth/19}
                \centering
                \includegraphics[width = \textwidth]{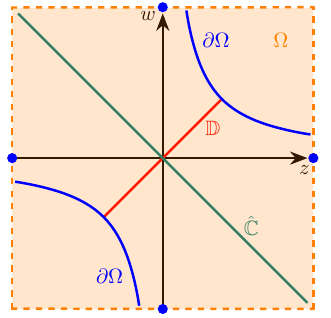}
        \end{minipage}
        \begin{minipage}{6\textwidth/19}
            \centering
            \includegraphics[width = \textwidth]{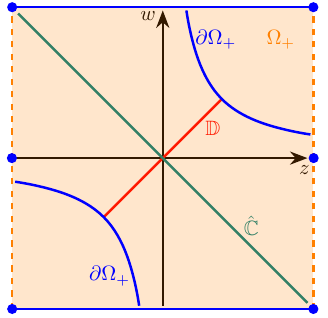}
        \end{minipage}
    \begin{minipage}{6\textwidth/19}
        \centering
        \includegraphics[width = \textwidth]{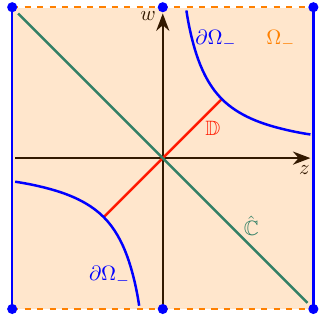}
    \end{minipage}
        \caption{Schematic picture of the domains $\Omega$ (left), $\Omega_+$ (center) and $\Omega_-$ (right) with points at infinity.}
        \label{fig:OmegaPM}
    \end{figure}

 The following proposition provides the first step towards the proof of Theorem \ref{thm:HOmegaDecomposition}. It shows that  $\Omega_+$ and $\Omega_-$ are  simply connected subdomains of $\Omega$ which can be mapped biholomorphically onto  $\C^2$ by means of the elementary maps
\begin{align*}
  & \Psi_{+} : \Omega_{+} \to \C^2 \, ,  \quad \Psi_+(z,w)=\left(\frac{z}{1-zw},w \right) \, , \\
  &  \Psi_{-} : \Omega_{-} \to \C^2 \, ,  \quad \Psi_-(z,w)=\left(z,\frac{w}{1-zw} \right) \, .
\end{align*}
In terms of the induced  composition operators
\begin{align} \label{eq:C+}
  \mathcal{C}_{\Psi_+} & : \mathcal{H}(\C^2) \to \mathcal{H}(\Omega_+) \, , \qquad \mathcal{C}_{\Psi_+} (F):=F \circ \Psi_+ \, , \\ \label{eq:C-}
   \mathcal{C}_{\Psi_-} & : \mathcal{H}(\C^2) \to \mathcal{H}(\Omega_-) \, , \qquad \mathcal{C}_{\Psi_-} (F):=F \circ \Psi_- \, ,
  \end{align}
  this means  that both Fr\'echet spaces, $\mathcal{H}(\Omega_+)$ and $\mathcal{H}(\Omega_-)$, are  isomorphic to $\mathcal{H}(\C^2)$.

\begin{proposition}[] \label{prop:OmegaPM}
  The domains $\Omega_{+}$ and $\Omega_-$ are biholomorphically equivalent to $\C^2$  and
  the composition operators (\ref{eq:C+}) and (\ref{eq:C-})  are continuous and bijective.
\end{proposition}

The proof is by  simple verification. The second ingredient for the proof of Theorem \ref{thm:HOmegaDecomposition} is
the canonical decomposition of the Fr\'echet space $\mathcal{H}(\C^2)$ of all entire functions $F :  \C^2 \to \C$,
\begin{equation*} \label{eq:HC2Decomposition}
  \mathcal{H}(\C^2)=\mathcal{H}_+(\C^2) \oplus \mathcal{H}_-(\C^2) \, ,
  \end{equation*}
  where $\mathcal{H}_+(\C^2)$ denotes the  closed linear hull  of the monomials   $z^p w^q$ with $p \ge q$,  and
  $\mathcal{H}_-(\C^2)$ is  the closed linear hull of the monomials $z^p w^q$ with $p < q$.
We call $\mathcal{H}_-(\C^2)$ the past and $\mathcal{H}_+(\C^2)$ the future in $\mathcal{H}(\C^2)$.
Roughly speaking, we  shall show that the direct sum of the restriction of the operator $\mathcal{C}_{\Psi_-}$ to $\mathcal{H}_-(\C^2)$
and the restriction of the operator $\mathcal{C}_{\Psi_+}$ to $\mathcal{H}_+(\C^2)$
provides a continuous linear bijection from $\mathcal{H}(\C^2)$ onto $\mathcal{H}(\Omega)$. Along the way, we will encounter several subtleties, which we carefully have to bypass. The construction is in two steps. In the first step, we establish the following Proposition, which provides
the bridge from the past and the future in $\mathcal{H}(\C^2)$ to the past and the future in $\mathcal{H}(\Omega)$.

 \begin{proposition} \label{prop:Extension} \phantom{aa}\\[-11mm]
   \begin{itemize}
     \item[(a)]
   For each $F^+ \in \mathcal{H}_+(\C^2)$ the function  $\mathcal{C}_{\Psi_+}(F^+)\in \mathcal{H}(\Omega_+)$ has a holomorphic extension $\Phi_+(F^+) \in \mathcal{H}_+(\Omega)$, and the induced linear operator
   $$ \Phi_+   : \mathcal{H}_+(\C^2) \to \mathcal{H}_+(\Omega)  $$
is  continuous and one-to-one.
\item[(b)]  For each  $F^- \in \mathcal{H}_-(\C^2)$ the function $\mathcal{C}_{\Psi_-}(F^-)\in \mathcal{H}(\Omega_-)$ has a holomorphic extension $\Phi_-(F^-) \in \mathcal{H}_-(\Omega)$, and the induced linear operator
   $$ \Phi_-: \mathcal{H}_-(\C^2) \to \mathcal{H}_-(\Omega)  $$
is  continuous and one-to-one.
\end{itemize}
\end{proposition}

In particular, the linear operator $\Phi:=\Phi_+\oplus \Phi_- : \mathcal{H}(\C^2) \to \mathcal{H}(\Omega)$ is continuous and one-to-one. In the second step, it remains to  prove that $\Phi$ is onto in order to complete the following diagram:
\begin{center}
\begin{figure}[H]
\begin{tikzcd}[column sep=10pt, row sep=50pt]
  \mathcal{H}(\C^2) \arrow[d, "\phantom{;;}\Phi_{\phantom{+}}" description, two heads]  &  = \arrow[d,dash,draw=white, "\phantom{;}=_{\phantom{a}}" description] & \mathcal{H}_+(\C^2) \arrow[d, "\phantom{;;}\Phi_+" description, two heads] & \oplus\arrow[d,dash,draw=white, "\hspace{0.15cm}\oplus_{\phantom{a}}" description] & \mathcal{H}_-(\C^2) \arrow[d, "\phantom{;;}\Phi_-" description, two heads] \\
   \mathcal{H}(\Omega)   &  = & \mathcal{H}_+(\Omega)  & \oplus & \mathcal{H}_-(\Omega)
\hide{  \arrow[d, "\Phi" description] & =   \arrow[d, "\Phi_+" description] &\oplus &  \arrow[d, "\Phi_-" description]\\}
\end{tikzcd}
\caption{Decomposition of $\mathcal{H}(\Omega)$}
\end{figure}
\end{center}

For the proof that $\Phi : \mathcal{H}(\C^2) \to \mathcal{H}(\Omega)$ is onto,  we show that the natural projections that map from $\mathcal{H}(\C^2)$ onto its closed complementary subspaces $\mathcal{H}_+(\C^2)$ and $\mathcal{H}_-(\C^2)$ also act as continuous projections from $\mathcal{H}(\Omega_\pm)$ onto $\mathcal{H}_\pm(\Omega)$. As we shall see, in  combination with Proposition \ref{prop:OmegaPM},  this will imply that  the linear operators $\Phi_\pm: \mathcal{H}_\pm(\C^2) \to \mathcal{H}_\pm(\Omega)$ in Proposition  \ref{prop:Extension} are surjective and also that $\mathcal{H}(\Omega)=\mathcal{H}_+(\Omega) \oplus \mathcal{H}_-(\Omega)$.

\medskip

Note that the bidisk $\D^2$ is contained in each of the domains $\Omega$, $\Omega_+$ and $\Omega_-$, see again Figure~\ref{fig:OmegaPM}.
We denote by $\pi_+ : \mathcal{H}(\D^2) \to \mathcal{H}_+(\D^2)$ and  $\pi_- : \mathcal{H}(\D^2) \to \mathcal{H}_-(\D^2)$ the natural projection maps defined by
\begin{equation} \label{eq:proj}
 \pi_+\left( \sum \limits_{p,q=0}^{\infty} a_{p,q} u^pv^q \right):= \sum \limits_{p \ge q} a_{p,q} u^pv^q \, , \qquad
 \pi_-\left( \sum \limits_{p,q=0}^{\infty} a_{p,q} u^pv^q \right):= \sum \limits_{p < q} a_{p,q} u^pv^q\, .
 \end{equation}
 Clearly, the restrictions of $\pi_{\pm}$ to the subspace $\mathcal{H}(\C^2)$ of $\mathcal{H}(\D^2)$ are continuous projections from $\mathcal{H}(\C^2)$ onto $\mathcal{H}_{\pm}(\C^2)$. The following result shows that the restrictions of $\pi_{\pm}$ to the subspaces $\mathcal{H}(\Omega_{\pm})$ of $\mathcal{H}(\D^2)$ are continuous projections from $\mathcal{H}(\Omega_{\pm})$ onto $\mathcal{H}_{\pm}(\Omega)$ as well.

 \begin{proposition} \label{prop:Extension2}\phantom{aa}\\[-10mm]
   \begin{itemize}
   \item[(a)] For each $f \in \mathcal{H}(\Omega_+)$ the function $\pi_+(f) \in \mathcal{H}(\D^2)$ has a holomorphic extension to a function in $\mathcal{H}_+(\Omega)$, and the induced linear operator
     $$\pi_+ : \mathcal{H}(\Omega_+) \to \mathcal{H}_+(\Omega)$$
     is continuous.
     Moreover,  $\pi_+ : \mathcal{H}(\Omega_+) \to \mathcal{H}(\Omega_+)$ is a projection onto $\mathcal{H}_+(\Omega)$, and the diagram in Figure \ref{fig:1} (a) commutes.
   \item[(b)] For each $f \in \mathcal{H}(\Omega_-)$ the function $\pi_-(f) \in \mathcal{H}(\D^2)$ has a holomorphic extension to a function in $\mathcal{H}_-(\Omega)$, and the  induced linear operator
     $$\pi_- : \mathcal{H}(\Omega_-) \to \mathcal{H}_-(\Omega)$$ is continuous. Moreover,
$\pi_- : \mathcal{H}(\Omega_-) \to \mathcal{H}(\Omega_-)$ is a projection onto $\mathcal{H}_-(\Omega)$,
   and the diagram in Figure \ref{fig:1} (b) commutes.
   \end{itemize}

   \begin{figure}[h]
\centering
\begin{subfigure}{.5\textwidth}
  \centering  \begin{tikzcd}[column sep=50pt]
\mathcal{H}(\C^2) \arrow[r, "\mathcal{C}_{\Psi_+}","\cong"'] \arrow[d, black,  "\pi_+", two heads]
& \mathcal{H}(\Omega_+) \arrow[d, "\pi_+", two heads] \\[30pt]
\mathcal{H}_+(\C^2) \arrow[r, "\Phi_+","\cong"' ]
& \mathcal{H}_+(\Omega)
\end{tikzcd}
  \caption{}
  \label{fig:sub1}
\end{subfigure}%
\begin{subfigure}{.5\textwidth}
  \centering  \begin{tikzcd}[column sep=50pt]
\mathcal{H}(\C^2) \arrow[r, "\mathcal{C}_{\Psi_-}","\cong"'] \arrow[d, black,  "\pi_-", two heads]
& \mathcal{H}(\Omega_-) \arrow[d, "\pi_-", two heads] \\[30pt]
\mathcal{H}_-(\C^2) \arrow[r, "\Phi_-","\cong"']
& \mathcal{H}_-(\Omega)
\end{tikzcd}
   \caption{}
  \label{fig:sub2}
\end{subfigure}
\caption{Commutation of composition operators and projections}
\label{fig:1}
\end{figure}
\end{proposition}

Note that we are systematically abusing language by using the same symbols, $\pi_+$ and $\pi_-$, for denoting the projections (\ref{eq:proj}) which act on different functions spaces depending on the context. This should cause no serious confusion.

\begin{remark}
  Proposition \ref{prop:Extension2} implies that the set $\mathcal{H}_+(\Omega)$ is also  closed w.r.t.~the Fr\'echet topology of the ambient space $\mathcal{H}(\Omega_+)$. Hence the map $\pi_+ : \mathcal{H}(\Omega_+) \to \mathcal{H}(\Omega_+)$ is a \textbf{continuous} projection onto the closed subspace
  $\mathcal{H}_+(\Omega)$ of $\mathcal{H}(\Omega_+)$. The same remark applies to $\pi_-$. However, for the proof of Theorem \ref{thm:HOmegaDecomposition} we need the full force of Proposition \ref{prop:Extension2}, that is, the continuity of $\pi_+ :  \mathcal{H}(\Omega_+) \to \mathcal{H}_+(\Omega)$ w.r.t.~the canonical Fr\'echet topology of $ \mathcal{H}_+(\Omega)$.
  \end{remark}

  It follows at once from Proposition \ref{prop:Extension2}  that the operators $\Phi_+  : \mathcal{H}_+(\C^2) \to \mathcal{H}_+(\Omega)$ and $ \Phi_-: \mathcal{H}_-(\C^2) \to \mathcal{H}_-(\Omega)$ in Proposition \ref{prop:Extension} are not only  continuous and one-to-one, but also surjective.
Their direct sum $\mathcal{C}_{\Psi_+} \oplus \mathcal{C}_{\Psi_-} : \mathcal{H}(\C^2)  \to \mathcal{H}(\Omega)$ is a continuous bijection:

\begin{theorem}[Canonical decomposition of $\mathcal{H}(\Omega)$, fine structure] \label{thm:HOmegaDecomposition2}
The map
    \begin{align*}
      \Phi\quad  :  & \quad    \mathcal{H}(\C^2) \to \mathcal{H}(\Omega) \, , \\
      & \quad   F \mapsto \Phi(F):=\Phi_+\left(\pi_+(F)\right) +\Phi_-\left(\pi_-(F)\right)
    \end{align*}
    is a continuous linear bijection. Its inverse map is given by
     \begin{align*}
      \Phi^{-1}\quad  :  & \quad    \mathcal{H}(\Omega) \to \mathcal{H}(\C^2) \, , \\
      & \quad   f \mapsto \Phi^{-1}(f):=\pi_+\left(f \circ \Psi_+^{-1}\right)+\pi_-\left(f \circ \Psi_-^{-1}\right) \, .
    \end{align*}
    In addition, the isomorphisms $\Phi : \mathcal{H}(\C^2) \to \mathcal{H}(\Omega)$ and $\Phi^{-1} : \mathcal{H}(\Omega) \to \mathcal{H}(\C^2)$ preserve the past and the future,
    $$ \Phi\left(\mathcal{H}_-(\C^2)\right)=\mathcal{H}_-(\Omega) \, , \qquad  \Phi\left(\mathcal{H}_+(\C^2)\right)=\mathcal{H}_+(\Omega) \, .$$
  \end{theorem}

  By construction, $\Phi$ maps each element $u^pv^q$ of the canonical Schauder basis $(u^pv^q)_{p,q}$ of $\mathcal{H}(\C^2)$ to $f_{p,q} \in \mathcal{H}(\Omega)$ defined in (\ref{eq:SchauderBasis}). Since $\Phi: \mathcal{H}(\C^2) \to \mathcal{H}(\Omega)$ is a continuous bijection, we finally arrive at the following result.

  \begin{corollary} \label{cor:Schauder}
    $(f_{p,q})_{p,q \in \N_0}$ is a Schauder basis of $\mathcal{H}(\Omega)$. In particular, each $f \in \mathcal{H}(\Omega)$ has a unique representation as
    $$ f(z,w)=\sum \limits_{p,q=0}^{\infty} a_{p,q} f_{p,q}(z,w) \, .$$
    This series converges absolutely and locally uniformly in $\Omega$, and the Schauder coefficients $a_{p,q}$ of $f$ are given by
    \begin{equation} \label{eq:SchauderCoeff}
      a_{p,q}= \begin{cases}
      \displaystyle   -\frac{1}{4 \pi^2} \int \limits_{\partial \D} \int
\limits_{\partial \D} f\left(z, \frac{w}{1+zw} \right) \frac{dzdw}{z^{p+1}
  w^{q+1}} & p<q \, , \\[2mm]
\displaystyle -\frac{1}{4 \pi^2} \int \limits_{\partial \D} \int
\limits_{\partial \D} f\left( \frac{z}{1+zw} ,w\right) \frac{dzdw}{z^{p+1}
  w^{q+1}} & p \ge q
        \end{cases}
      \end{equation}

    \end{corollary}
Clearly,  $(f_{p,q})_{p<q}$ is a Schauder basis of $\mathcal{H}_-(\Omega)$ and
$(f_{p,q})_{q \le p}$ is a Schauder basis of $\mathcal{H}_+(\Omega)$.

\begin{remark}[Terminology]
  We briefly indicate why one might call $\mathcal{H}_-(\Omega)$ the past and $\mathcal{H}_+(\Omega)$ the future in $\mathcal{H}(\Omega)$. As described in the introduction, one motivation for our study of the Fr\'echet space
  $\mathcal{H}(\Omega)$ comes from the fact that the restriction to the diagonal yields the Fr\'echet algebra
  $\mathcal{A}(\D)$ for which~a continuous star product has been constructed in \cite{KrausRothSchoetzWaldmann}. In view of Theorem \ref{thm:HOmegaDecomposition}, the Fr\'echet algebra $\mathcal{A}(\D)$
  decomposes in
  $$ \mathcal{A}(\D)=\mathcal{A}_+(\D) \oplus \mathcal{A}_-(\D) \, $$ where
  each element $\phi$ in $\mathcal{A}_+(\D)$ has the form
  $$ \phi(r e^{it}) =\sum \limits_{p \ge q} a_{p,q} \frac{r^{p+q}}{(1-r^2)^p} e^{i (p-q) t} \, , \quad z=r e^{it} \in \D \, .$$
  Hence $t \mapsto \phi(r e^{it})$ has no negatively indexed Fourier coefficients and thus
  belongs to the Hardy space $H^2(\partial \D)$, see e.g.~\cite{HelsonSarason1967}.
\end{remark}

\section{Automorphisms and the invariant Laplacian of $\Omega$} \label{sec:Automorphisms}

The  complex manifold $\Omega$ has a rich and highly symmetric structure. In particular, $\Omega$ is a Stein manifold with the holomorphic density property, see \cite[Sec.~3]{KrausRothSchleissingerWaldmann}.
Roughly speaking, the density property implies by way of Andersen--Lempert theory \cite{Forstneric2011} that  $\Omega$ has a \textit{very} large automorphism group $\Aut(\Omega)$. In particular, $\Aut(\Omega)$ is infinite--dimensional. It is therefore of interest to characterize those subgroups of $\Aut(\Omega)$, whose elements have meaningful geometric or algebraic characteristic properties.
For instance, since one can think of $\Omega$ as a submanifold of the complex projective space ${\C}\operatorname{P}^3$, it is possible to speak about  \textit{rational} automorphisms of $\Omega$, and these have been completely characterized in \cite{Lamy2005}. In this section we give a characterization of those automorphisms of $\Omega$, which leave the Laplacian $\Delta_{zw}$ on $\Omega$ invariant in the following, very weak, sense.

\begin{definition}
  Let $O\subseteq \Omega$ be a non--empty open set and let  $T \colon O \to \Omega$  be a holomorphic mapping. We say that the operator $\Delta_{zw}$ is $T$--\textbf{invariant} if
  \begin{equation}           \label{eq:InvarianceHolomorphic}
    \Delta_{zw}(F \circ T)
    =
    \Delta_{zw}(F) \circ T \quad \text{for every $F \in \mathcal{H}(\Omega)$.}
\end{equation}
 \end{definition}

As it has been pointed out in Section \ref{sec:prelim}, the operator $\Delta_{zw}$ is $T$--invariant for every automorphism $T$ belonging to the M\"obius subgroup $\mathcal{M}$, see (\ref{eq:OmegaAutomorphismIntro}).
Since $\Aut(\Omega)$ is much larger than $\mathcal{M}$, one might ask if there exist other automorphisms $T$ for which $\Delta_{zw}$ is $T$--invariant. The following result shows that this is not the case: the M\"obius automorphisms of $\Omega$ are the \textit{only} locally defined holomorphic maps on $\Omega$, for which the Laplacian $\Delta_{zw}$ is invariant.

\begin{theorem} \label{thm:Invariance}
  Let  $O$ be a subdomain of $\Omega$, and let  $T \colon O \to \Omega$  be a holomorphic mapping. Then the following conditions are equivalent:
  \begin{itemize}
  \item[(a)] $\Delta_{zw}$ is $T$--invariant.
    \item[(b)] $T \in \mathcal{M}$.
    \end{itemize}
  \end{theorem}

In short: The M\"obius--invariant Laplacian is $T$--invariant if and only if $T$ is  M\"obius.

\medskip

Note that in Theorem \ref{thm:Invariance} we only assume  $T : O \to \Omega$ is defined in \textit{some} subdomain $O$ of~$\Omega$. It is to emphasize that the $T$--invariance condition (\ref{eq:InvarianceHolomorphic}) is very weak, since
we only assume this condition holds for functions $F$ which are defined and holomorphic on \textit{all} of $\Omega$ and not for functions which are holomorphic merely on some open subset of $\Omega$, see Remark \ref{rem:TestFunctions} below.
The crucial point is to make sure that $\mathcal{H}(\Omega)$ contains sufficiently many functions to ensure that condition (\ref{eq:InvarianceHolomorphic}) forces $T \in \mathcal{M}$.
In fact, condition (\ref{eq:InvarianceHolomorphic}) cannot be weakened any further, since there is no open set in $\hat{\C}^2$ which is strictly larger than $\Omega$ and carries  nonconstant holomorphic functions:

  \begin{theorem}   \label{thm:Invariance2}
    Let $F \in \mathcal{H}(\Omega)$. Suppose $\xi \in \partial \Omega$ and there is an open neighborhood $U \subseteq \hat{\C}^2$ of $\xi$ such that $F$ is bounded on $U$. Then $F$ is constant.
    \end{theorem}

    The proof of the non--trivial implication ``(a) $\Longrightarrow$ (b)'' of Theorem \ref{thm:Invariance} will be given in Section~\ref{sec:proof2}. It is rather long and is therefore divided into several steps which might be of independent interest. We therefore briefly outline the idea of the proof of Theorem \ref{thm:Invariance}. The situation becomes somewhat simpler on the second configuration space $G$ of $\hat{\C}$, see Section~\ref{sub:conf}. By  mapping $\Omega$ biholomorphically onto $G$ as explained in Section \ref{sub:conf}, it is easy to check that Theorem \ref{thm:Invariance} is equivalent to

    \begin{theorem}[Invariance of the Laplacian on the second configuration space] \label{thm:InvarianceConf}
      Let $O$ be a subdomain of $G$ and $T=(T_1,T_2) : O \to G$ a holomorphic mapping. Then the following conditions are equivalent:
      \begin{itemize}
      \item[(a)] (The Laplacian on $G$ is $T$--invariant)\\
        For every $F \in \mathcal{H}(G)$ and all $(z,w) \in O$,
        \begin{equation} \label{eq:aut1}
  (z-w)^2 \partial_z\partial_w \left(F \circ T\right)(z,w)=\left(T_1(z,w)-T_2(z,w)\right)^2 (\partial_z\partial_w F)(T(z,w))
\end{equation}
\item[(b)] There exists  ${\psi \in \Aut(\hat{\C})}$ such that  $T(z,w)=\left(\psi(z),\psi(w)\right)$
or $T(z,w)= \left(\psi(w),\psi(z)\right)$.
        \end{itemize}
      \end{theorem}

      The crucial  part of the proof of Theorem \ref{thm:InvarianceConf} is to show that any holomorphic mapping $$T : O \to G \, , \quad  (z,w) \mapsto T(z,w)=\begin{pmatrix} T_1(z,w) \\ T_2(z,w) \end{pmatrix} \, , $$
      for which the $T$--invariance property (\ref{eq:aut1}) holds has  ``separated variables'' in the sense that either $T_1$ depends only on $z$ and $T_2$ depends only on $w$,  or vice versa. By carefully choosing appropriate ``test functions'' $F \in \mathcal{H}(G)$, we shall show as a first step that 
      the \textit{difference} $T_1(z,w)-T_2(z,w)$ can be expressed as  the \textit{difference} $H_1(z)-H_2(w)$ of a holomorphic function $H_1$ depending only on $z$ and another holomorphic function $H_2$ depending only on  $w$. In addition,  the $T$--invariance property (a) in Theorem \ref{thm:InvarianceConf} then translates into an ODE--condition for $H_1$ and $H_2$:
     \begin{proposition}  \label{prop:Invariance2}
      Let  $O\subseteq G$ be a non--empty open set and let  $T \colon O \to G$  such that the $T$--invariance condition (\ref{eq:aut1}) holds. Then there exist subdomains $D_1,D_2$ of $\C$ with $D_1 \times D_2 \subseteq O$ and holomorphic functions $H_1 \in \mathcal{H}(D_1)$, $H_2 \in \mathcal{H}(D_2)$ such that
      $T_1(z,w)-T_2(z,w)=H_1(z)-H_2(w)$ and
      \begin{equation} \label{eq:schwarzeq1}
                 \left( \frac{H_1(z)-H_2(w)}{z-w} \right)^2
        =
        H_1'(z)
        H_2'(w)
           \end{equation}
   for all $(z,w) \in D_1 \times D_2$.
 \end{proposition}

In a next step, we show that equation (\ref{eq:schwarzeq1}) puts severe constraints on the functions $H_1$ and~$H_2$:

\begin{proposition} \label{prop:Invariance3}
 Let $D_1, D_2$ be open sets in $\C$. Suppose $H_1 \in \mathcal{H}(D_1)$ and ${H_2 \in \mathcal{H}(D_2)}$ satisfy condition (\ref{eq:schwarzeq1}) for all $(z,w) \in D_1 \times D_2$.
  Then  either $H_1,H_2 \in \Aut(\hat{\C})$ or $H_1, H_2$ are constant. In both cases,  $H_1=H_2$.
  \end{proposition}

\begin{remark} \label{rem:HawleySchiffer}
As we shall explain in Remark \ref{rem:Moeb} (a), Proposition \ref{prop:Invariance3} can be seen as a generalization of a result of Hawley and Schiffer \cite{HawleySchiffer1966}. Their result is the special case of Proposition \ref{prop:Invariance3} when $D_1=D_2$ and $H_1=H_2$ is assumed from the outset. Note however that for applying Proposition \ref{prop:Invariance3} in the proof of Theorem \ref{thm:InvarianceConf}, we need $D_1 \times D_2 \subseteq G$ and hence $D_1 \cap D_2=\emptyset$. In particular, the result of Hawley and Schiffer does not suffice for proving Theorem \ref{thm:InvarianceConf}.
\end{remark}

Taking Propositions \ref{prop:Invariance2} and \ref{prop:Invariance3} together shows that for a holomorphic mapping $T : O \to G$ with  the $T$--invariance property (\ref{eq:aut1}) there is $H_1 \in \Aut(\hat{\C})$ such that
\begin{equation} \label{eq:cond12}
    T_1(z,w)-T_2(z,w)=H_1(z)-H_1(w) \, \quad \text{ for all } (z,w) \in O \, .
    \end{equation}
The main trick now is to realize that the change of variables 
 $(z,w) \mapsto \left(H^{-1}_1(z),H^{-1}_2(w)\right)$  is permissible, which simplifes condition (\ref{eq:cond12})  to
$$  T_1(z,w)-T_2(z,w)=z-w \, , $$
or -- equivalently -- the $T$--invariance condition (\ref{eq:aut1}) becomes
\begin{equation} \label{eq:test}
  \partial_z\partial_w (F \circ T)(z,w)=\left(\partial_z\partial_w F \right)\left(T(z,w) \right) \, .
  \end{equation}
Again choosing appropriate ``test functions'' $F \in \mathcal{H}(G)$, we finally show that (\ref{eq:test}) implies $T_1$ is either independent of $z$ or independent of $w$. From here, it is then immediate that $T_1$ and $T_2$ are Möbius.

\begin{remark}\label{rem:Invariance}
  If we restrict in Theorem \ref{thm:Invariance} to the ``diagonal'' $\{(z,\overline{z}) \, : \, z \in \D\} \subset \Omega$ which we identify with the unit disk $\D$ and to the ``rotated diagonal''
    {$\{(z,-\overline{z}) \, : \, z \in \hat{\C}\} \subset \Omega$} which we identify with the Riemann sphere $\hat{\C}$, we get:
      \begin{itemize}
    \item[(i)]
  Let $O$ be a subdomain of $\D$, and $T : O \to \D$ holomorphic. Then  the following conditions are equivalent:
  \begin{itemize}
  \item[(a)] $\Delta_{\D}$ is $T$--invariant, that is,  $\Delta_{\D}\left( f \circ T \right)=\Delta_\D f \circ T$
    for all $f \in \A(\D)$.
    \item[(b)] $T \in \Aut(\D)$.
    \end{itemize}
 \item[(ii)]
  Let $O$ be a subdomain of $\hat{\C}$, and $T : O \to \hat{\C}$ holomorphic. Then  the following conditions are equivalent:
  \begin{itemize}
  \item[(a)] $\Delta_{\hat{\C}}$ is $T$--invariant, that is,  $\Delta_{\hat{\C}}\left( f \circ T \right)=\Delta_{\hat{\C}} f \circ T$
    for all $f \in \A(\hat{\C})$.
    \item[(b)] $T \in \Rot(\hat{\C})$.
    \end{itemize}
    \end{itemize}
Similar to Theorem \ref{thm:Invariance}, in both cases $X=\D$ or $X=\hat{\C}$, condition (a) is fairly weak in the sense that we assume $\Delta_{X} (f \circ T)=\Delta_{X}f \circ T$ only for functions $f$ belonging to the set $\A(X)$, which is~a comparatively small subset of the set of all real--analytic  functions on $X$.
  \end{remark}

\section{Automorphisms and a  canonical  holomorphic Riemannian metric on  $\Omega$} \label{sec:RiemannianMetric}

Returning to Theorem \ref{thm:Invariance} we now add a third equivalent condition in pure differential geometric terms.
For this purpose we define a holomorphic analogue $g_\Omega$ of the hyperbolic and spherical Riemannian metric on $\Omega$ in local coordinates by
\begin{equation}
    \label{eq:MetricOmega}
    g_\Omega(z,w)
    =
    \frac{1}{(1-zw)^2}
    dz \vee dw \, .
\end{equation}
 It is easy to check that formula (\ref{eq:MetricOmega}) defines a  non-vanishing and holomorphic section. In the literature, such a section is referred to as \emph{holomorphic Riemannian metric} on $\Omega$, see e.g.~\cite{Manin1988}.
It is not difficult to see that the (complex--valued) Riemannian Laplacian corresponding to $g_\Omega$ is exactly $\Delta_{zw}$.

\begin{theorem}\label{thm:AutOmega}
Let $O$ be a subdomain of $\Omega$, and let $T: O \to \Omega$ be a holomorphic mapping. Then the following conditions are equivalent:
   \begin{enumerate}
       \item[(a)] $T^*g_\Omega = g_\Omega$ as two--forms.
       \item[(b)] $T \in \mathcal{M}$.
   \end{enumerate}
\end{theorem}

Note that condition (a) of Theorem \ref{thm:Invariance} is considerably weaker than condition (a) of Theorem~\ref{thm:AutOmega}: the invariance of the two-form is a local condition, whereas \eqref{eq:InvarianceHolomorphic} merely poses a condition on \emph{globally} defined holomorphic functions.

\begin{remark}\label{rem:Invariance2}
Let $$\lambda_{\D}(z):=\frac{1}{1-|z|^2} \qquad \text{ and } \qquad \lambda_{\hat{\C}}(z)=\frac{1}{1+|z|^2}$$ denote the density of the Poincar\'e metric on $\D$ and  the density of the spherical metric on $\hat{\C}$.
  If we restrict in Theorem \ref{thm:AutOmega} to the ``diagonal'' $\{(z,\overline{z}) \, : \, z \in \D\} \subset \Omega$ which we identify with the unit disk $\D$ resp.~to the ``rotated diagonal''
    {$\{(z,-\overline{z}) \, : \, z \in \hat{\C}\} \subset \Omega$} which we identify with the Riemann sphere $\hat{\C}$, we get
      \begin{itemize}
    \item[(i)]
  Let $O$ be a subdomain of $\D$ and $T : O \to \D$ holomorphic. Then  the following conditions are equivalent:
  \begin{itemize}
  \item[(a)] $T^*\lambda_{\D}=\lambda_{\D}$.
    \item[(b)] $T \in \Aut(\D)$.
    \end{itemize}
 \item[(ii)]
  Let $O$ be a subdomain of $\hat{\C}$ and $T : O \to \hat{\C}$ holomorphic. Then  the following conditions are equivalent:
  \begin{itemize}
  \item[(a)] $T^*\lambda_{\hat{\C}}=\lambda_{\hat{\C}}$.
    \item[(b)] $T \in \Rot(\hat{\C})$.
    \end{itemize}
    \end{itemize}
  \end{remark}

\section{Proofs of Propositions \ref{prop:Extension} and \ref{prop:Extension2}, Theorems \ref{thm:HOmegaDecomposition} and \ref{thm:HOmegaDecomposition2}, and Corollary \ref{cor:Schauder}} \label{sec:proof}

\subsection*{Proof of Proposition \ref{prop:Extension}}
We only prove part (a) since the proof of (b) merely requires some obvious modifications.
For every  $F^+ \in \mathcal{H}_+(\C^2)$ there are coefficients $a_{p,q}(F^+) \in \C$ such that
\begin{equation} \label{eq:F+}
  F^+(u,v)=\sum \limits_{p \ge q} a_{p,q}(F^+) u^p v^q \, ,
  \end{equation}
  and the series converges uniformly and absolutely for $(u,v)$ in  any compact subset of $\C^2$.
  Since $\Psi_+ : \Omega_+ \to \C^2$ is biholomorphic, it follows from Proposition \ref{prop:OmegaPM} that
  \begin{equation} \label{eq:SeriesHOmega+}
     \mathcal{C}_{\Psi_+}(F^+)(z,w)= \left(F^+ \circ \Psi_+\right)(z,w)=\sum \limits_{p \ge q} a_{p,q}(F^+) \frac{z^pw^q}{(1-zw)^p} \,
      \end{equation}
  for all $(z,w) \in \Omega_+$, and the convergence is absolute and locally uniform in $\Omega_+$.
  Note that each individual term of the series in (\ref{eq:SeriesHOmega+}) has the form $a_{p,q}(F^+) f_{p,q}$ and is therefore a holomorphic function
  on the larger domain $\Omega$. Our goal is to prove the following.

\begin{proposition} \label{prop:ExtensionA}
 Every point $(z_0,w_0) \in \Omega$ has  a neighborhood $U \subseteq \Omega$  with the following properties:
\begin{itemize}
\item[(i)] for each $F^+ \in \mathcal{H}_+(\C^2)$ of the form (\ref{eq:F+}) the series in (\ref{eq:SeriesHOmega+}) converges absolutely and uniformly in $U$ to some $f^+ \in \mathcal{H}(U)$;
\item[(ii)] if $(F^+_n) \subseteq \mathcal{H}_+(\C^2)$ converges locally uniformly in $\C^2$ to $F^+\in \mathcal{H}_+(\C^2)$, then $(f^+_n)$ converges uniformly in $U$ to $f^+$.
\end{itemize}
\end{proposition}

Hence for each $F^+ \in \mathcal{H}_+(\C^2)$ the holomorphic function $\mathcal{C}_{\Psi_+}(F^+) \in \mathcal{H}(\Omega_+)$ has a holomorphic extension $\Phi_+(F^+) \in \mathcal{H}_+(\Omega)$ given by the series in (\ref{eq:SeriesHOmega+}), and this series  converges locally uniformly in  $\Omega$. Moreover, the mapping $\Phi_+ : \mathcal{H}_+(\C^2) \to \mathcal{H}_+(\Omega)$ is continuous and linear, and is explicitly given by
\begin{align} \label{eq:Phi+}
   \Phi_+ \left( \sum \limits_{p \ge q} a_{p,q} u^p v^q \right) &:=\sum \limits_{p \ge q} a_{p,q} f_{p,q} \, .
\end{align}
Clearly,  $\Phi_+ : \mathcal{H}_+(\C^2) \to \mathcal{H}_+(\Omega)$  is  one--to--one. Hence we have reduced Proposition \ref{prop:Extension} to Proposition \ref{prop:ExtensionA}.

\begin{proof}[Proof of Proposition \ref{prop:ExtensionA}]
    Throughout the proof, we write
    $${\norm{f}}_K \coloneqq \max_{(z,w) \in K} \abs{f(z,w)}$$
    for $K \subseteq \Omega$ compact and continuous functions $f \colon K \to \C$. Fix a compact set $K \subseteq \Omega$ and consider the compact subsets $K_1 \coloneqq \{(z,w) \in K \colon \abs{w} \le 1\}$ as well as $K_2 \coloneqq \{(z,w) \in K \colon \abs{w} \ge 1\}$. Here, we define $\abs{\infty} \coloneqq \infty$. By continuity of the functions
    \begin{equation*}
        f_{1,0}(z,w)
        =
        \frac{z}{1-zw}
        \quad \text{and} \quad
        f_{1,1}(z,w)
        =
        \frac{zw}{1-zw}
    \end{equation*}
    we have
    \begin{equation*}
        M
        \coloneqq
        \max
        \left\{
            1, \,
            {\norm{f_{1,0}}}_{K_1}, \,
            {\norm{f_{1,1}}}_{K_2}
        \right\}
        <
        \infty\,.
    \end{equation*}
    Let $B \coloneqq \{(z,w) \in \C^2 \,\colon \,\abs{z} \le 2M,\, \abs{w} \le 2M\}$. As $F^+ \in \mathcal{H}_+(\C^2)$, the Cauchy estimates show
    \begin{equation*}
        \abs[\big]
        {a_{p,q}(F^+)}
        \le
        \frac{{\norm{F^+}}_B}{(2M)^{p+q}}
        \, , \qquad
        p \ge q.
    \end{equation*}
    Using the submultiplicativity of the norms $\norm{\argument}_{K_1}$ and $\norm{\argument}_{K_2}$, we estimate
    \begin{align*}
        \sum_{p \ge q}
        \abs{a_{p,q}(F^+)}
        \cdot
        {\norm{f_{p,q}}}_{K_1}
        \le
        \sum_{p \ge q}
        \abs{a_{p,q}(F^+)}
        \cdot
        {\norm{f_{1,0}}}^p_{K_1}
        \cdot
        {\norm{w^q}}_{K_1}
        \le
        \sum_{p \ge q}
        \frac{\norm{F^+}_B}{(2M)^{p+q}}
        M^p
        \le
        \frac{8}{3} {\norm{F^+}}_B
    \end{align*}
    as well as
    \begin{align*}
        \sum_{p \ge q}
        \abs{a_{p,q}(F^+)}
        \cdot
        {\norm{f_{p,q}}}_{K_2}
        \le
        \sum_{p \ge q}
        \abs{a_{p,q}(F^+)}
        \cdot
        {\norm{f_{1,1}}}^p_{K_2}
        \cdot
        {\norm{w^{q-p}}}_{K_2}
        \le
        \sum_{p \ge q}
        \frac{{\norm{F^+}}_B}{(2M)^{p+q}}
        M^{p}
        \le
        \frac{8}{3} {\norm{F^+}}_B \,.
    \end{align*}
    As $K = K_1 \cup K_2$, this implies
    \begin{equation}
        \label{eq:ExtensionAProof}
        \sum_{p \ge q}
        \abs{a_{p,q}(F^+)}
        \cdot
        {\norm{f_{p,q}}}_{K}
        \le
        \frac{8}{3} {\norm{F^+}}_B
        <
        \infty \, .
    \end{equation}
    Thus, the series \eqref{eq:SeriesHOmega+} converges uniformly and absolutely on $K$. Varying $K$, we see that the series \eqref{eq:SeriesHOmega+} converges uniformly on compact subsets of $\Omega$ to a limit $f^+ \in \mathcal{H}(\Omega)$. Finally, note that every $(z_0,w_0) \in \Omega$ has a compact neighbourhood $K$ in $\Omega$. This completes the proof of the first statement by setting $U \coloneqq K$. Let now $(F^+_n) \subseteq \mathcal{H}_+(\C^2)$ converge locally uniformly in $\C^2$ to $F^+\in \mathcal{H}_+(\C^2)$. By linearity of the coefficient functionals $a_{p,q}$, the estimate \eqref{eq:ExtensionAProof} implies
    \begin{equation*}
        \norm[\big]
        {f^+ - f_n^+}_U
        \le
        \sum_{p \ge q}
        \abs{a_{p,q}(F^+- F^+_n)}
        \cdot
        {\norm{f_{p,q}}}_{U}
        \le
        4 {\norm{F^+ - F^+_n}}_B \, ,
        \qquad
        n \in \N.
    \end{equation*}
    As $B \subseteq \C^2$ is compact, $(f^+_n)$ therefore converges uniformly to $f^+$ on $U$.
\end{proof}

\subsection*{Proof of Proposition \ref{prop:Extension2}}

The proof of Proposition \ref{prop:Extension2} is based on Proposition \ref{prop:Extension} and the following
simple observation.

\begin{proposition} \label{prop:pi}
  Let $p,q \in \N_0$. Then
  \begin{align*}
    \pi_+ \left( \frac{z^p w^q}{(1-zw)^p} \right)& =\begin{cases} \displaystyle \frac{z^p w^q}{(1-zw)^p} & \text{ if } p \ge q \\
      0 & \text{ if } p<q \end{cases}\\[2mm]
          \pi_- \left( \frac{z^p w^q}{(1-zw)^p} \right)& =\begin{cases} 0 & \text{ if } p \ge q \\
            \displaystyle \frac{z^p w^q}{(1-zw)^p} & \text{ if } p <q \end{cases}
    \end{align*}
  \end{proposition}

  \begin{proof}
This follows from
    \[\frac{z^p w^q}{(1-zw)^p}=\begin{cases} \displaystyle \frac{ (zw)^q}{(1-zw)^p} z^{p-q} \in \mathcal{H}_+(\D^2)& \text{ if } p \ge q \\[4mm]
 \displaystyle  \frac{ (zw)^p}{(1-zw)^p} w^{q-p} \in \mathcal{H}_-(\D^2) & \text{ if } p<q \, .\qedhere\end{cases}
    \]
    \end{proof}

    \begin{proof}[Proof of Proposition \ref{prop:Extension2}]
      It suffices to prove Part (a). Let $f \in \mathcal{H}(\Omega_+)$. Then, by Proposition~\ref{prop:OmegaPM},  $F:=f \circ \Psi_+^{-1} \in \mathcal{H}(\C^2)$, so
      $$ F(u,v)=\sum \limits_{p,q=0}^{\infty} a_{p,q} u^p v^q \, , \quad (u,v) \in \C^2 \, , $$
      or, equivalently,
      $$ f(z,w)=\sum \limits_{p,q=0}^{\infty} a_{p,q}  \frac{z^pw^q}{(1-zw)^p} \, , \qquad (z,w) \in \Omega_+ \, .$$
      This series converges in particular absolutely and locally uniformly in $\D^2$, so
      $$ \pi_+(f)(z,w)=\sum \limits_{p,q=0}^{\infty} a_{p,q} \pi_+ \left(  \frac{z^pw^q}{(1-zw)^p} \right)=
      \sum \limits_{p \ge q} a_{p,q} f_{p,q}(z,w) \, , \qquad (z,w) \in \D^2 \, $$
      by Proposition \ref{prop:pi}.
      Recall that by slight abuse of notation, we denote both projections $\mathcal{H}(\C^2) \to \mathcal{H}_+(\C^2)$ and $\mathcal{H}(\Omega_+) \to \mathcal{H}_+(\Omega)$ by the same symbol $\pi_+$. In view of (\ref{eq:Phi+}), this implies
\begin{equation} \label{eq:comm}
  \pi_+(f)=\Phi_+(\pi_+(F))=\left( \mathcal{C}_{\Psi_+} \circ \pi_+ \circ \mathcal{C}_{\Psi_+}^{-1} \right) f \, .
  \end{equation}
      Note that the right--hand side is a composition of the three continuous operators
      $$ \mathcal{C}^{-1}_{\Psi_+} : \mathcal{H}(\Omega_+) \to \mathcal{H}(\C^2)\, , \quad
\pi_+ : \mathcal{H}(\C^2) \to \mathcal{H}_+(\C^2)\, , \quad  \mathcal{C}_{\Psi_+} : \mathcal{H}_+(\C^2) \to \mathcal{H}_+(\Omega) \, ,
$$
see Proposition \ref{prop:OmegaPM} and Proposition \ref{prop:Extension}. Thus $\pi_+ : \mathcal{H}(\Omega_+) \to \mathcal{H}_+(\Omega)$ is continuous and (\ref{eq:comm}) also shows that the diagram in Figure \ref{fig:1} (a) commutes. Using (\ref{eq:comm}) twice and the fact that $\pi_+ : \mathcal{H}(\C^2) \to \mathcal{H}_+(\C^2$) is a projection implies
$$\pi_+(\pi_+(f))=\pi_+\circ  (\mathcal{C}_{\Psi_+} \circ \pi_+ \circ \mathcal{C}_{\Psi_+}^{-1}) f= (\mathcal{C}_{\Psi_+} \circ \pi_+ \circ \pi_+ \circ \mathcal{C}_{\Psi_+}^{-1})f=
(\mathcal{C}_{\Psi_+} \circ \pi_+ \circ \mathcal{C}_{\Psi_+}^{-1})f=\pi_+(f)$$
for all $f \in \mathcal{H}(\Omega_+)$, so $\pi_+ : \mathcal{H}(\Omega_+) \to \mathcal{H}_+(\Omega) $ is a projection as well. It remains to prove that $\pi_+(f)=f$ for all $f \in \mathcal{H}_+(\Omega)$. However, recalling that $\mathcal{H}_+(\Omega)$ is the closed linear span of the functions $f_{p,q}$ with $p \ge q$ in $\mathcal{H}(\Omega)$, this follows at once from Proposition \ref{prop:pi} since $\pi_+(f_{p,q})=f_{p,q}$ whenever $p \ge q$.
              \end{proof}

\subsection*{Proof of Theorem \ref{thm:HOmegaDecomposition} and Theorem \ref{thm:HOmegaDecomposition2}}

Proposition \ref{prop:OmegaPM} and Proposition \ref{prop:Extension} imply that the direct sum $\Phi:=\Phi_+ \oplus \Phi_- : \mathcal{H}_+(\C^2) \oplus \mathcal{H}_-(\C^2) \to \mathcal{H}(\Omega)$ is linear, continuous and injective. To prove its surjectivity, let  $f \in \mathcal{H}(\Omega)$. Then
\[F^+:=\pi_+(f \circ \Psi_+^{-1}) \in \mathcal{H}_+(\C^2)\quad\text{and}\quad F^-:=\pi_-(f \circ \Psi_-^{-1}) \in \mathcal{H}_-(\C^2)\, .\] Therefore, $f^+:=\Phi_+(F^+) \in \mathcal{H}_+(\Omega)$ and $f^-:=\Phi_-(F^-) \in \mathcal{H}_-(\Omega)$, and we deduce
\begin{align*}
\Phi(F^++F^-) &=\Phi_+(F^+)+\Phi_-(F^-) =\Phi_+\left(\pi_+(f \circ \Psi_+^{-1})\right)+\Phi_-\left(\pi_-(f \circ \Psi_-^{-1}) \right)\\ & =
\pi_+ \left( \mathcal{C}_{\Psi_+}(f \circ \Psi_{+}^{-1}) \right)+\pi_- \left( \mathcal{C}_{\Psi_-}(f \circ \Psi_{-}^{-1}) \right)
\end{align*}
in view of the commutation relations of Figure \ref{fig:1}. Thus
\[\Phi(F^++F^-)=\pi_+(f)+\pi_-(f)=f\,,\]
and $\Phi$ is surjective. Note that
\[F^++F^-=\pi_+(f \circ \Psi_+^{-1})+\pi_-(f \circ \Psi_-^{-1})=\Phi^{-1}(f)\,.\]
The proof of Theorem \ref{thm:HOmegaDecomposition2} (and Theorem \ref{thm:HOmegaDecomposition}) is complete.\hfill{$\square$}

\subsection*{Proof of Corollary \ref{cor:Schauder}}

Let $f \in \mathcal{H}(\Omega)$. By Theorem \ref{thm:HOmegaDecomposition2}, $F:=\Phi^{-1}(f) \in \mathcal{H}(\C^2)$, so
$$ F(u,v)=\sum \limits_{p,q=0}^{\infty} a_{p,q} u^p v^q \, $$ with
\begin{equation} \label{eq:TaylorCoeff}
  a_{p,q}=\frac{1}{(2\pi i)^2} \int \limits_{\partial\D} \int \limits_{\partial\D} \frac{F(z,w)}{z^{p+1} w^{q+1}} \, dz dw \, .
\end{equation}
  We can write
$$ F(u,v)=\pi_+(F)(u,v)+\pi_-(F)(u,v)=\sum \limits_{p \ge q}  a_{p,q} u^p v^q+\sum \limits_{p < q}  a_{p,q} u^p v^q \, .
$$
Hence, again by Theorem \ref{thm:HOmegaDecomposition2},
\begin{align*}
  f(z,w)&=\Phi(F)(z,w)=\Phi_+(\pi_+(F))(z,w)+\Phi_-(\pi_-(F))(z,w)\\ &  =\sum \limits_{p \ge q}  a_{p,q} f_{p,q}(z,w)+\sum \limits_{p < q}  a_{p,q} f_{p,q}(z,w) \, .
          \end{align*}
In the last step, we have used identity (\ref{eq:Phi+}) for $\Phi_+$  and the corresponding identity for $\Phi_-$.
Since both series converge absolutely and locally uniformly in $\Omega$, we can rearrange the order of summation.
The explicit formula for the Schauder coefficients $a_{p,q}$ in terms of $f$ follows immediately from (\ref{eq:TaylorCoeff}). Indeed, if $p \ge q$, then $a_{p,q}$ is the corresponding Taylor coefficient of $\pi_+(F)$ and
$$ \pi_+(F)=\pi_+\left( f \circ \Psi_+^{-1}\right)=\pi_+(f) \circ \Psi_+^{-1} \, ,
$$
in view of Figure \ref{fig:1}(a). On the other hand, $\pi_+(f) \circ \Psi_+^{-1} \in \mathcal{H}_+(\C^2)$, so the Cauchy integral formula shows that
$$ a_{p,q}=\frac{1}{(2\pi i)^2} \int \limits_{\partial\D} \int \limits_{\partial\D} \frac{\pi_+(F)(z,w)}{z^{p+1} w^{q+1}} \, dz dw =\frac{1}{(2\pi i)^2} \int \limits_{\partial\D} \int \limits_{\partial\D} \frac{(f \circ \Psi_+^{-1}) (z,w)}{z^{p+1} w^{q+1}} \, dz dw $$
Inserting the explicit expression for $\Psi_+^{-1}$ proves (\ref{eq:SchauderCoeff}) for $p \ge q$. The proof for $p<q$ is identical.~\hfill{$\square$}

\section{Proofs  of Propositions \ref{prop:Invariance2} and \ref{prop:Invariance3},  Theorems \ref{thm:Invariance} and  \ref{thm:Invariance2}, and  Theorem \ref{thm:AutOmega}} \label{sec:proof2}

\subsection*{Proof of Proposition \ref{prop:Invariance2}}
Let  $T=(T_1,T_2) \colon O \longrightarrow G$ be a holomorphic mapping defined on some non--empty open subset  $O$ of $G$ such that
\begin{equation} \label{eq:aut1a}
  (z-w)^2 \partial_z\partial_w \left(F \circ T\right)(z,w)=\left(T_1(z,w)-T_2(z,w)\right)^2 (\partial_z\partial_w F)(T(z,w))
  \end{equation}
  for all $F \in \mathcal{H}(G)$ and all $(z,w) \in O$.
We shall make use of condition (\ref{eq:aut1a}) by inserting several carefully chosen specific functions ${F \in \mathcal{H}(G)}$.
\medskip

First, we insert
$$ F \in \mathcal{H}(G) \, , \quad F(z,w)=\frac{1}{z-w}$$
into (\ref{eq:aut1a}). By an elementary, but lengthy computation we find that condition (\ref{eq:aut1a}) for this choice of $F$ is equivalent to
\begin{equation} \label{eq:aut2}
  \begin{split}
2 \left( \frac{T_1(z,w)-T_2(z,w)}{z-w} \right)^2 & =\left( T_1-T_2\right) \left( \partial_z\partial_w T_1-\partial_z\partial_wT_2 \right)\\ & \qquad \qquad -2 \left( \partial_w T_1-\partial_w T_2 \right) \left( \partial_z T_1-\partial_z T_2 \right)
 \end{split} \end{equation}
  If we take
  $$ F \in \mathcal{H}(G) \, , \quad F(z,w)=\frac{1}{(z-w)^2}$$
in (\ref{eq:aut1a}), then another computation implies that  condition (\ref{eq:aut1a}) for this choice of $F$ is equivalent to
\begin{equation} \label{eq:aut3}
  \begin{split}
3 \left( \frac{T_1(z,w)-T_2(z,w)}{z-w} \right)^2 & =\left( T_1-T_2\right) \left( \partial_z\partial_w T_1-\partial_z\partial_wT_2 \right)\\ & \qquad \qquad -3 \left( \partial_w T_1-\partial_w T_2 \right) \left( \partial_z T_1-\partial_z T_2 \right)
\end{split}
\end{equation}
Together, the two conditions (\ref{eq:aut2}) and (\ref{eq:aut3}) clearly  imply
$$ \left( T_1-T_2\right) \left( \partial_z\partial_w T_1-\partial_z\partial_wT_2 \right)=0 \, .$$
Since $T=(T_1,T_2)$ maps the open set $O \subseteq G$ into $G$, we get  $T_1(z,w)\not=T_2(z,w)$ for all $(z,w) \in O$ by definition of $G$, and hence
$$ \partial_z\partial_w (T_1 - T_2) = 0 \quad \text{on } O \, .$$
Consequently,  all terms with mixed derivatives  in the Taylor expansion of $T_1 - T_2$ about each $(z_0, w_0) \in O$ vanish. This way, we obtain auxiliary functions $H_1 \in \mathcal{H}(D_1)$, $H_2 \in \mathcal{H}(D_2)$ for certain  domains  $D_1, D_2 \subseteq \C$ with $D_1 \times D_2 \subseteq O$  such that
\begin{equation*}
    T_1(z,w) - T_2(z,w)
    =
    H_1(z) - H_2(w),
    \qquad
    (z,w) \in D_1 \times D_2 \, .
  \end{equation*}
  Note that we can choose $D_1,D_2\subseteq\C$ since $O$ and $G \cap \C^2$ always have a nonempty intersection. Inserting this expression into (\ref{eq:aut2}), it is easily checked that condition (\ref{eq:schwarzeq1}) holds for $H_1$ and $H_2$, and the proof of Proposition \ref{prop:Invariance2} is complete.
  \hfill{$\square$}

\subsection*{Proof of Proposition \ref{prop:Invariance3}}

 If $H_2$ is constant on $D_2$, then (\ref{eq:schwarzeq1}) implies $H_1$ is constant on $D_2$ and $H_1=H_2$.
    If $H_2$ is  not constant on $D_2$, then  (\ref{eq:schwarzeq1}) implies $H_1$ is not constant on $D_1$ either and in fact  $H_1$ has vanishing Schwarzian derivative away from possible zeros of $H_1'$, that is,
     \begin{equation*} \label{eq:schwarz1}
         \Schwarzian_{H_1}(z)
         :=
         \bigg(
            \frac{H_1''(z)}{H_1'(z)}
         \bigg)'
         -
         \frac{1}{2}
         \bigg(
            \frac{H_1''(z)}{H_1'(z)}
         \bigg)^2
         =
         0\, .
       \end{equation*}
       This follows by fixing $w \in D_2$ and differentiating (\ref{eq:schwarzeq1}) twice w.r.t.~$z \in D_1$.
    It is a classical fact (see \cite{laine1992}) that $\Schwarzian_{H_1}=0$ is equivalent to  $H_1 \in \Aut(\hat{\C})$.
Switching  the roles of $H_1$ and $H_2$, we see that either $H_1$, $H_2$ are both constant or $H_1,H_2 \in \Aut(\hat{\C})$. Finally, if $H_1 \in \Aut(\hat{\C})$ and $H_2 \in \Aut(\hat{\C})$, then
the right--hand side of \eqref{eq:schwarzeq1} is never zero. Hence the left--hand side is never zero either, and thus $H_1(z) \neq H_2(w)$ for all $z,w \in \C$ such that $z\not=w$. This is clearly only possible if $H_1=H_2$. \hfill{$\square$}

\begin{remark}  \label{rem:Moeb} \phantom{aa}\\[-10mm]
    \begin{enumerate}
    \item[(a)] As noted earlier in Remark \ref{rem:HawleySchiffer}, Proposition \ref{prop:Invariance3} is an extension of a result of Hawley and Schiffer \cite{HawleySchiffer1966}, who considered the special case $D_1=D_2$.
 We therefore briefly discuss this case a bit further in order to put Proposition \ref{prop:Invariance3} into a somewhat broader context. Note that if $D_1=D_2$ or merely $D_1\cap D_2 \not=\emptyset$,  it follows immediately from (\ref{eq:schwarzeq1}) that $H_1 \equiv H_2$. Let $H:=H_1$. Excluding the trivial case that $H$ is constant, condition (\ref{eq:schwarzeq1}) implies $H$ is univalent in $D_1$. Therefore, the function
   \begin{equation*}
     \hat{\Schwarzian}_H  : G_1 \to \C \, , \quad
     \hat{\Schwarzian}_{H}(z,w)
        :=
        \begin{cases}
          \displaystyle  \frac{H'(z)H'(w)}{\left(H(z)-H(w)\right)^2}
           -
           \frac{1}{(z-w)^2}
           \; &\text{for} \;
           z \neq w \\[4mm]
           \Schwarzian_H(z)
           \; &\text{for} \;
           z = w \\
        \end{cases}
      \end{equation*}
      is well--defined and it is not difficult to show that $\hat{\Schwarzian}_H$ is holomorphic on $D_1 \times D_2$.
      A discussion of this fact and its implications can be found in  \cite{HawleySchiffer1966}.
      The difference between  the point of view of Hawley--Schiffer \cite{HawleySchiffer1966} and the one we pursue in this paper is now apparent: While \cite{HawleySchiffer1966} extends the Schwarzian derivative $\Schwarzian_H$, a function  of one variable,  to~a holomorphic function  $\hat{\Schwarzian}_H$ of two variables $(z,w)$   in a neighborhood of the ``diagonal'' $(z,z)$, our approach to Theorem \ref{thm:InvarianceConf} is based on Proposition \ref{prop:Invariance3} for the opposite case $D_1 \cap D_2=\emptyset$, so we exclude the diagonal.

    \item[(b)] Fixing $w \in D_2$, equation \eqref{eq:schwarzeq1} is a Riccati differential equation for the unknown function $H_1 : D_1 \to \C$. It is an elementary fact that each Riccati equation can be transformed in~a canonical way into normal form $u'=A(z)+u^2$, see  (see \cite[p.~165]{laine1992}). For the specific Riccati equation (\ref{eq:schwarzeq1}) the canonical transformation is given by
      $$ H_1(z)=H_2'(w) (z-w)^2 u(z)+H_2(w)+\frac{H_2'(w)}{2} (z-w) $$
      and turns the Riccati equation (\ref{eq:schwarzeq1}) into
      $$ u'=u^2 \, ,$$
      whose set of solutions is given by $\{(z+c)^{-1} \, : \, c \in \C\} \cup \{0\}$. Elementary considerations then again lead to the conclusion that either $H_1$ is constant or belongs to $\Aut(\hat{\C})$. This provides a slightly different approach to Proposition \ref{prop:Invariance3}.
      \end{enumerate}
\end{remark}

\subsection*{{Proof of Theorem \ref{thm:InvarianceConf}} ``(a) $\Longrightarrow$ (b)''}
Let  $T=(T_1,T_2) \colon O \longrightarrow G$ be  a holomorphic mapping defined on some subdomain  $O$ of $G$ such that condition (\ref{eq:aut1}) holds for all ${F \in \mathcal{H}(G)}$ and all $(z,w) \in O$.

\medskip

(i)
By Proposition \ref{prop:Invariance2} and Proposition \ref{prop:Invariance3} there exists a function $H:=H_1 : \hat{\C} \to \hat{\C}$, which is either constant or belongs to $\Aut(\hat{\C})$,  such that
\begin{equation} \label{eq:H1=H2}
  \left( \frac{H(z)-H(w)}{z-w} \right)^2=H'(z) H'(w)
  \end{equation}
and
\begin{equation} \label{eq:T1-T2}
  T_1(z,w)-T_2(z,w)=H(w)-H(z) \,
  \end{equation}
  for all $(z,w) \in O$. Since $T$ maps into $G$, we have $T_1(z,w)\not=T_2(z,w)$ for all $(z,w) \in O$, so the case $H$ is constant cannot occur, and therefore  $H \in \Aut(\hat{\C})$.

  \medskip

  (ii) We claim, without loss of generality, that  we may assume
  \begin{equation} \label{eq:assumption}
    T_1(z,w)-T_2(z,w)=w-z \quad \text{ for all } (z,w) \in O \, .
    \end{equation}

    To prove this claim, consider
    $$S(z,w):=\left( H^{-1}(z),H^{-1}(w) \right) \, , \qquad (z,w) \in G \, .$$
    Note that $S$ is a biholomorphic self--map of $G$  and  $\tilde{O}:=S^{-1}(O)$ is a subdomain of $G$.
    The holomorphic mapping
    $$ \tilde{T}:=T \circ S : \tilde{O} \to G \,  $$
    then satisfies
    $$ \tilde{T}_1(z,w)-\tilde{T}_2(z,w)= T_1\left(S(z,w)\right)-T_2\left(S(z,w)\right)=H\left(H^{-1}(w) \right)-H\left( H^{-1}(z) \right)  \, $$
    in view of (\ref{eq:T1-T2}), and hence
    \begin{equation} \label{eq:T1-T2a}
      \tilde{T}_1(z,w)-\tilde{T}_2(z,w)=w-z \quad \text{ for all } (z,w) \in \tilde{O} \, .
      \end{equation}
    To finish the proof of Theorem \ref{thm:InvarianceConf}  it therefore suffices to show that the Laplacian on $G$ is $\tilde{T}$--invariant, that is,  condition (\ref{eq:aut1}) holds with $T$ replaced by $\tilde{T}$. To check this, observe that
    \begin{equation} \label{eq:InvarSchwarzEq1}
      \begin{split}
      (z-w)^2 \partial_z\partial_w \left(F \circ \tilde{T}\right)(z,w)& =(z-w)^2 \partial_z\partial_w \left( \left(F \circ T \right) \circ S\right)(z,w)\\
                                                                      &=(z-w)^2 \left( \partial_z\partial_w  \left(F \circ T \right) (S(z,w))  \right) \left( H^{-1} \right)'(z)  \left( H^{-1} \right)'(w)
                                                                    \end{split}
                                                                        \end{equation}
                                                                        by the chain rule and the particular form of $S(z,w)=\left( H^{-1}(z),H^{-1}(w) \right)$. Since $T$ satisfies condition (\ref{eq:aut1}), we have
      \begin{equation} \label{eq:InvarSchwarzEq2}
\begin{split}
\partial_z\partial_w  \left(F \circ T \right) (S(z,w))
  & =\left( \frac{H^{-1}(z)-H^{-1}(w)}{ H^{-1}(z)-H^{-1}(w)}\right)^2 \partial_z\partial_w  \left(F \circ T \right) (S(z,w)) \\
  &=   \left( \frac{ \tilde{T}_1(z,w)-\tilde{T_2}(z,w)}{ H^{-1}(z)-H^{-1}(w)}\right)^2  \left( \partial_z\partial_w F \right) \left( \tilde{T}(z,w) \right) \\
  &=   \left( \frac{z-w}{ H^{-1}(z)-H^{-1}(w)}\right)^2  \left( \partial_z\partial_w F \right) \left( \tilde{T}(z,w) \right)
    \,
  \end{split}
  \end{equation}
where we have taken (\ref{eq:T1-T2a}) into account in the last identity. Being a M\"obius transformation, $H^{-1}$ has the property
$$   \left( \frac{z-w}{ H^{-1}(z)-H^{-1}(w)}\right)^2 =\frac{1}{\left( H^{-1} \right)'(z)  \left( H^{-1} \right)'(w)} \, . $$
This is exactly condition (\ref{eq:H1=H2}), but for the inverse $H^{-1}$ of $H$ instead of $H$.
  Therefore, combining (\ref{eq:InvarSchwarzEq1}) and (\ref{eq:InvarSchwarzEq2}), leads to
\begin{align*}
  (z-w)^2 \partial_z\partial_w \left(F \circ \tilde{T}\right)(z,w) &= (z-w)^2 \left( \partial_z\partial_w F \right) \left( \tilde{T}(z,w) \right)\\ & =\left( \tilde{T}_1(z,w)-\tilde{T}_2(z,w) \right)^2  \left( \partial_z\partial_w F \right) \left( \tilde{T}(z,w) \right)
  \end{align*}
  again making use of (\ref{eq:T1-T2a}) in the last step. Consequently, (\ref{eq:aut1}) holds with $T$ replaced by $\tilde{T}$.

  \medskip

  (iii) In view of (ii) it remains to prove Theorem \ref{thm:InvarianceConf} under the additional assumption that $T$ satisfies (\ref{eq:assumption}).

We therefore return  to (\ref{eq:aut1}), where we choose
$$ F \in \mathcal{H}(G) \, , \quad F(z,w)=\frac{z}{(z-w)^2} \, .$$
It is easily seen that equation  (\ref{eq:aut1}) then takes  the form
\begin{equation} \label{eq:aut4}
\frac{2
   \left(\partial_w T_1(z,w)-\partial_z T_1(z,w)-1\right)}{w-z}+\partial_z\partial_w T_1(z,
 w)=0 \, ,
 \end{equation}
With
$$ F \in \mathcal{H}(G) \, , \quad F(z,w)=\frac{z^2}{(z-w)^2} \, ,$$
equation  (\ref{eq:aut1}) has the form
$$ 2 \partial_wT_1(z,w) \partial_z T_1(z,w)+2 T_1(z,w) \left(\frac{2
   \left(\partial_w T_1(z,w)-\partial_z T_1(z,w)-1\right)}{w-z}+\partial_z\partial_w T_1(z,
 w)\right)=0 \, .$$
Taken together with (\ref{eq:aut4}), this implies
$$ (\partial_wT_1(z,w)) (\partial_z T_1(z,w))=0 \, ,$$
and either $T_1$ is independent of $z$ or independent of $w$. Going back to (\ref{eq:aut4}), we see that either
$T_1(z,w)=w+\gamma$ or $T_1(z,w)=z+\gamma$ for some $\gamma \in \C$. In a similar way, we can deduce that either
$T_2(z,w)=w+\gamma'$ or $T_2(z,w)=z+\gamma'$ for some $\gamma' \in \C$.
Finally, since $T_1(z,w)-T_2(z,w)=w-z$, we see that
$$ \text{ either} \quad T(z,w)=(z+\gamma,w+\gamma) \quad \text{ or } \quad T(z,w)=(w+\gamma,z+\gamma) $$
for some $\gamma \in \C$. The proof is complete. \hfill{$\square$}

\begin{remark} \label{rem:TestFunctions} The preceding proof would become much simpler under the much stronger assumption that (\ref{eq:aut1}) holds for all $F$ holomorphic merely in a neighborhood of the origin in $\C^2$, since then one could choose the monomials $z^n w^m$, $n,m=0,1,\ldots$ as ``test functions'' instead of functions holomorphic on all of $\Omega$ resp.~on $G=\conf$.
\end{remark}

\hide{In particular, pullback with the inverse of $T(z,w) = (z,1/w)$ maps eigenfunctions of $\Delta_{zw}$ to eigenfunctions of $\Delta_G$. By virtue of Remark~\ref{rem:Moeb}[c], the identity \eqref{eq:Moeb} thus takes the form
\begin{equation}
    H_1'(z) H_2'(w)
    =
    \frac{\big(1 - H_1(z)H_2(w)\big)^{2}}{(1 - zw)^2},
    \qquad
    (z,w) \in D_1 \times D_2 \subseteq \Omega
\end{equation}
and any two such functions are either constant or $H_2 = 1/z \circ H_1 \circ 1/z \in \Aut(\hat{\C})$.}

\subsection*{Proof of Theorem \ref{thm:Invariance2}}
Let $\xi=(z_0,w_0) \in \partial \Omega$ and choose $r>0$ such that $K_r(z_0) \times K_r(w_0) \subseteq U$. For fixed $w_1 \in K_r(w_0)$ consider the holomorphic
function $$F_{w_1}:\hat{\C}\setminus \{1/w_1\} \to\C \, , \quad  z\mapsto F(z,w_1)\, .$$
By assumption,  $F_{w_1}$ is bounded on $\hat{\C}\setminus \{1/w_1\}$, and thus by Liouville's theorem there is a constant $c(w_1)\in  \C$ such that  $F_{w_1}(z)=c(w_1)$ for all $z\in\hat{\C}\setminus \{1/w_1\}$. A similar argument shows  that for each $z_1 \in K_r(z_0)$ there is a constant $\tilde{c}(z_1)$ such that $F_{z_1}(w):=F(z_1,w)=\tilde{c}(z_1)$ for all $w\in\hat{\C}\setminus\{1/z_1\}$. Notice that for $(z_1,w_1) \in U \cap \Omega$, we have
$c(w_1)=F(z_1,w_1)=\tilde{c}(z_1)$. This clearly implies that $F$ is constant in $U \cap \Omega$, and hence on $\Omega$.
\hide{
We even have $c(w_0)=\tilde{c}(z_0)\coloneqq c$ since $F_{w_0}(z_0)=F_{z_0}(w_0)$.

	Since $F$ is holomorphic in a neighbourhood $U$ of $(z_0,w_0)$ (or bounded in $U\cap\Omega$) and $\del\Omega$ is the graph $(z,1/z)$, we find $(z_1,w_1)\in U\cap\del\Omega$. As for $(z_0,w_0)$, we find that $F_{w_1}(z)=c_1(w_1)$ for all $z\in\hat{\C}$, $F_{z_1}(w)=F(z_1,w)=\tilde{c}_1(z_1)$ for all $w\in\hat{\C}$ and $c_1(w_1)=\tilde{c}_1(z_1)\coloneqq c_1$. By construction, we also find $(\tilde{z},\tilde{w}),(\hat{z},\hat{w})\in U\cap\Omega$ such that $c=F_{w_0}(\tilde{z})=F_{z_1}(\hat{w})=c_0$.

	Since $(z_1,w_1)$ was arbitrarily chosen, we conclude that $F$ is constant on $U$ and thus on $\Omega$ by the identity principle.} \hfill{$\square$}

\subsection*{Proof of Theorem \ref{thm:AutOmega}}
    We only prove the non--trivial implication ``(a) $\Longrightarrow$ (b)''. Note again that $O$ and $\C^2 \cap \Omega$ always have a nonempty intersection, i.e. we may work in the standard chart and assume $O \subseteq \C^2$. Let $T(z,w) = (T_1(z,w), T_2(z,w))$. Suppressing the arguments, we compute
    \begin{align*}
        T^* g_{\Omega}
        \at{(z,w)}
        &=
        \big(1 - T_1T_2\big)^{-2}
        d(T_1)
        \vee
        d(T_2) \\
        &=
        \big(1 - T_1T_2\big)^{-2}
        \big(
            \partial_z T_1 dz
            +
            \partial_w T_1 dw
        \big)
        \vee
        \big(
            \partial_z T_2 dz
            +
            \partial_w T_2 dw
        \big) \\
        &=
        \big(1 - T_1T_2\big)^{-2}
        \Big(
            \big(
                (\partial_z T_1) (\partial_w T_2)
                +
                (\partial_w T_1) (\partial_z T_2)
            \big)
            dz \vee dw \\
            &+
            (\partial_z T_1)(\partial_w T_1) dz^2
            +
            (\partial_z T_2) (\partial_w T_2) dw^2
        \Big).
    \end{align*}
    Comparing coefficients with \eqref{eq:MetricOmega} yields first that $(\partial_z T_1)(\partial_w T_1) = 0 = (\partial_z T_2) (\partial_w T_2)$. That is, each component of $T$ only depends on $z$ or $w$. If $T$ depended on $z$ only, then $T^* dw = 0$, which is absurd in view of $T^*g_\Omega = g_\Omega$. Consequently, we may assume $T_1(z,w) = T_1(z)$ and ${T_2(z,w) = T_2(w)}$. Making this more precise, we fix $(z_0, w_0) \in O$ and define the holomorphic auxiliary functions
    \begin{align*}
        &H_1
        \colon
        O_1
        =
        \big\{
            z \in \C
            \;\big|\;
            (z, w_0)
            \in
            O
        \big\}
        \longrightarrow
        \hat{\C}, \quad
        H_1(z)
        =
        T_1(z, w_0) \\
        \textrm{and} \quad
        &H_2
        \colon
        O_2
        =
        \big\{
            w \in \C
            \;\big|\;
            (z_0,w)
            \in O
        \big\}
        \longrightarrow
        \hat{\C}, \quad
        H_2(z)
        =
        T_2(z_0, w).
    \end{align*}
    As $O$ is open, the same is true for its nonempty projections $O_1$ and $O_2$. By construction, we have $H_1'(z) = \partial_z T_1(z,w_0)$ and $H_2'(w) = \partial_w T_2(z_0,w)$ for every $(z,w) \in O_1 \times O_2$. Comparing the remaining coefficients this way yields
    \begin{equation*}
        H_1'(z) H_2'(w)
        =
        \frac{\big(1 - H_1(z)H_2(w)\big)^{2}}{(1 - zw)^2},
        \qquad
        (z,w) \in O_1 \times O_2.
    \end{equation*}
    This is the $\Omega$--version of \eqref{eq:schwarzeq1}, hence
    Proposition \ref{prop:Invariance3} now implies $H_1,H_2 \in \Aut(\hat{\C})$ and ${H_1(z)=1/H_2(1/z)}$ for all $z \in \hat{\C}$, as both functions $H_1,H_2$ being constant would imply $T^*g_\Omega = 0$. By construction, this means
    \begin{equation*}
        \big(T_1(z,w), T_2(z,w)\big)
        =
        \big(H_1(z), H_2(w)\big)
        =
        \big(H_1(z), 1/H_1(1/w)\big)
    \end{equation*}
    for all $(z,w) \in \Omega$, that is, $T \in \mathcal{M}$, see again \eqref{eq:OmegaAutomorphismIntro}. \hfill{$\square$}

\section*{Acknowledgements}

The authors thank Daniela Kraus for useful discussions about Schwarzian derivatives and for suggesting the alternative approach for proving Proposition \ref{prop:Invariance3} which is outlined in Remark~\ref{rem:Moeb}(b).

\end{document}